\documentclass{article}
\title{Lagrangian Dual Sections: A Topological Perspective on Hidden Convexity}
\author{
    Venkat Chandrasekaran\footnote{California Institute of Technology, Department of Computing and Mathematical Sciences and Department of Electrical Engineering, \texttt{venkatc@caltech.edu}}
    \and
    Timothy Duff\footnote{University of Missouri-Columbia, Department of Mathematics, \texttt{tduff@missouri.edu}}
    \and
    Jose Israel Rodriguez\footnote{University of Wisconsin-Madison, Department of Mathematics \texttt{jose@math.wisc.edu}}
    \and
    Kevin Shu\footnote{California Institute of Technology, Department of Computing and Mathematical Sciences, \texttt{kshu@caltech.edu}}
}

\usepackage[dvipsnames]{xcolor}
\usepackage{framed}
\usepackage{amsmath, amsthm, amsfonts, amssymb}

\usepackage{amssymb,amsmath}\usepackage{enumerate}
\usepackage{hyperref}
\usepackage[capitalize]{cleveref}
\usepackage{enumitem}
\usepackage[dvipsnames]{xcolor}
\usepackage{cancel}
\usepackage[all, cmtip]{xy}
\usepackage{tikz}
\usetikzlibrary{shapes}
\usepackage{svg}
\usepackage{import}
\usepackage{optidef}

\usepackage[utf8]{inputenc} % allow utf-8 input
\usepackage[T1]{fontenc}    % use 8-bit T1 fonts
\usepackage{hyperref}       % hyperlinks
\usepackage{url}            % simple URL typesetting
\usepackage{booktabs}       % professional-quality tables
\usepackage{amsfonts}       % blackboard math symbols
\usepackage{nicefrac}       % compact symbols for 1/2, etc.
\usepackage{microtype}
\usepackage{multirow}
\usepackage{xfrac}
\usepackage{titlesec}
\usepackage{caption}
\usepackage{subcaption}
\usepackage{amsmath} 
\usepackage{float}
\usepackage{algorithm}
\usepackage{algpseudocode}
\usepackage{multicol}

\titleformat{\subsubsection}
{\normalfont\normalsize\bfseries}{\thesubsubsection}{1em}{}
\usepackage{graphicx}

\usepackage{appendix}
\numberwithin{equation}{section}
\newcommand{\tr}{\mathrm{Tr}}

\newcommand{\cB}{\mathcal{B}}

\newcommand{\inclu}[0] {\ar@{^{(}->}}

\newcommand{\diag}{{\rm diag}}
\newcommand{\dist}{{\rm dist}}

\newcommand{\R}{\mathbb{R}}
\newcommand{\Z}{\mathbb{Z}}
\newcommand{\N}{\mathbb{N}}

\newcommand{\cN}{\mathcal{N}}

\newcommand{\PP}{\mathbb{P}}

\newcommand{\SO}{\text{SO}}

\DeclareMathOperator*{\conv}{conv}

\newcommand{\Diag}{\text{Diag}}

\newcommand{\argmin}{\operatornamewithlimits{argmin}}
\newcommand{\argmax}{\operatornamewithlimits{argmax}}

\newtheorem{thm}{Theorem}[section]
\newtheorem{theorem}{Theorem}[section]
\newtheorem{lemma}{Lemma}[section]
\newtheorem{corollary}{Corollary}[section]
\newtheorem{definition}{Definition}[section]

\newtheorem{proposition}[thm]{Proposition}

\newtheorem{remark}{Remark}

\newtheorem{example}{Example}[section]

\usepackage{mathtools}

\usepackage[T1]{fontenc}
\usepackage{lmodern}

 \usepackage[parfill]{parskip}

\newcommand{\cL}{\mathcal{L}}
\newcommand{\Rplus}{\R_{0+}}
\newcommand{\cM}{{\mathcal{M}}}

\renewcommand{\cB}{\mathcal{B}}

\newcommand{\Msing}{\cM^{(n,m)}_{\sigma}}

\newcommand{\cC}{\mathcal{C}}

\newcommand{\Voptf}{V_{f}}

\newcommand{\MR}{ \Rplus^{k+1} \times \mathcal{M} }  %%% codomain
\newcommand{\subspace}{\mathcal{S}}

\DeclareMathOperator{\St}{St}

\DeclareMathOperator{\Gr}{Gr}

\algrenewcommand\algorithmicrequire{\textbf{Input:}}
\algrenewcommand\algorithmicensure{\textbf{Output:}}

\newcommand{\demph}[1]{\emph{{\color{RoyalBlue}#1}}}
\usepackage{framed}

\newcommand{\joseCut}[1]{}

\begin{document}
\maketitle
\begin{abstract}
    Hidden convexity is a powerful idea in optimization: under the right transformations, nonconvex problems that are seemingly intractable can be solved efficiently using convex optimization. 
    We introduce the notion of a \emph{Lagrangian dual section} of a nonlinear program defined over a topological space, and we use it to give a sufficient condition for a nonconvex optimization problem to have a natural convex reformulation.
    We emphasize the topological nature of our framework, using only continuity and connectedness properties of a certain Lagrangian formulation of the problem to prove our results.  We demonstrate the practical consequences of our framework in a range of applications and by developing new computational methodology.
    First, we present families of nonconvex problem instances that can be transformed to convex programs in the context of linear inverse spectral problems -- which include quadratically constrained quadratic optimization and Stiefel manifold optimization as special cases -- as well as unbalanced Procrustes problems.
    In each of these applications, we both generalize prior results on hidden convexity and provide unifying proofs.  For the case of the linear inverse spectral problems, %our results are related to so-called Radon-Hurwitz numbers from algebraic topology and
    we also present a Lie-theoretic approach that illustrates connections with the Kostant convexity theorem.
    Second, we present novel algorithmic ideas that can be used to find globally optimal solutions to both Lagrangian forms of an optimization problem as well as constrained optimization problems when the underlying topological space is a Riemannian manifold.
\end{abstract}

\section{Introduction}

In global optimization, a major distinction has been made between `convex' and `nonconvex' optimization problems for decades \cite{rockafellar1993lagrange}.
While there is a rich theory for how to solve convex optimization problems to global optimality, the analogous theories for nonconvex optimization are distinctly weaker and more ad hoc.
However, convexity is not invariant under changes of variables, meaning that a problem may be convex in one presentation, but not in another.
Hence, it is of interest to identify situations in which a given problem can be \emph{reparametrized} in such a way that the resulting problem is a convex optimization problem in the new parameters. The phenomenon of problems having such a convex reparametrization is also referred to as `hidden convexity' in the literature.

In studying this phenomenon, it is helpful to identify properties of convex optimization problems which will continue to hold under reparametrization.  One such property is a tendency for \emph{parametric families} of convex  problems -- such as those derived from a Lagrangian formulation of a convex nonlinear program -- to have solutions which vary continuously with the parameter.  This property is prominently exploited by a number of optimization algorithms such as interior-point methods and augmented Lagrangian methods \cite{nocedal2006numerical}.  In particular, these procedures solve a given convex program by instead solving a sequence of parametrized problems by slowly varying the underlying parameter, and leveraging the continuity of the associated solution map to approach the solution of the original program.

In this paper, we show a sort of converse to this property -- if there exist continuously-varying solutions to a certain parametric family of problems associated to a constrained (not necessarily convex) nonlinear program, then we can provide a natural convex reformulation of the constrained program. In particular, this property can be defined over an arbitrary topological space, as the definition only involves topological aspects of the program. We show the practical utility of our results by exhibiting many examples satisfying our criterion arising from fundamental optimization problems, such as quadratically constrained quadratic optimization problems and Stiefel manifold optimization problems. Our methods yield new results on the exactness of semidefinite programming (SDP) relaxations for these problems. We also give applications to the less well-understood unbalanced Procrustes problem. 

We will formalize the notion of a continuous family of solutions to an optimization problem using the notion of a \demph{bundle}, which is ubiquitous in topology and geometry. To a constrained (possibly nonconvex) optimization problem, we associate a \demph{Lagrangian dual bundle}, which describes how the maximizer of a Lagrangian formulation of the problem varies as the Lagrange multipliers change. We then show how properties of that bundle can have geometric consequences for convex reparametrizations and the development of effective algorithms.

\subsection{Preliminaries and Main Result}
We give here the general setting of our analysis, and state our main result. We fix a topological space $\cM$ to be the domain of our optimization problem.
As examples, $\cM$ may be $\R^n$ with the Euclidean topology, a manifold, or an algebraic variety.
For our purposes, a constrained optimization problem on $\cM$ is specified by a continuous function $f : \cM \rightarrow \R^{k+1}$ for some $k \in \N$ and a constraint vector $c \in \R^k$:
\begin{equation}
     \label{eq:constrained} 
    V_f(c) = \max \{f_0(x) : f_1(x) = c_1, \dots, f_k(x) =  c_k, x \in \cM\}.\tag{opt-Main}
\end{equation}
We refer to $V_f$ as the \demph{optimal value function} of $f$. We assign the value $-\infty$ to $\Voptf(c)$ in case $\{x \in \cM : f_1(x) = c_1, \dots, f_k(x) =  c_k\}$ is empty.

\newcommand{\tDomain}{\Rplus^{k+1}}
\newcommand{\tDomainSmooth}{\R^{k+1}}

The continuous function 
$f:\cM\rightarrow \R^{k+1}$ also induces the~\demph{Lagrangian}: 
\begin{align}
\cL:\tDomain\times \cM\to\mathbb{R},\quad  \cL(t,x) = \langle t, f(x)\rangle, 
\end{align}
where $\langle \cdot, \cdot \rangle$ denotes the usual inner product on Euclidean space
and 
\[\tDomain := \{t \in \R^{k+1} : t_0 \ge 0, t \neq 0\}.\]
The  \demph{Lagrangian dual bundle} is defined as the following subset of the domain of $\mathcal{L}$: 
\begin{equation}
    \label{eq:bundle}
    \cB := \left\{(t,x) : x \in \argmax_{z \in \cM} \langle t, f(z)\rangle \right\} \subseteq \MR.
\end{equation}
We define a \demph{Lagrangian dual section of $\cB$} to be a continuous function $D : \tDomain \rightarrow \cM$ with  the property that $(t,D(t)) \in \cB$ for all $t \in \tDomain$. That is, for each $t \in \tDomain$, $D(t)$ is a maximizer of the function \begin{equation}
\ell_t:\cM\to \mathbb{R},\quad
\ell_t(x)=\langle t, f(x)\rangle,
\end{equation}
for all $t \in \tDomain$.\footnote{In other areas of mathematics, the data of a bundle would also involve a map $\pi : \cB \rightarrow \tDomain$, and a section would be right inverse of $\pi$. Our setup is equivalent, but simplifies the concept for our purposes.
In \cite{kevinthesis}, a related notion was defined, where it was termed `continuous maximization'.} 
If a Lagrangian dual section exists, we say the function $f$ is \demph{Lagrangian dual sectioned}, or simply \demph{sectioned}. 

Our main theorem shows a connection between the property of $f$ being sectioned and a convex reparametrization of the optimization problem defined by $V_f$:

\begin{theorem}
    \label{thm:hidden_cvx}
    If $\cM$ is compact and  $f : \cM \rightarrow \R^{k+1}$ is Lagrangian dual sectioned, then
    the optimal value function $\Voptf$ is concave.
\end{theorem}
We prove this result in \cref{subsec:main_proof}.
An easy corollary of this result is the existence of a reformulation of the optimization problem \eqref{eq:constrained} as a convex optimization problem; we will include a proof in \cref{app:proofs} for completeness.

\begin{corollary}
    \label{cor:convex_reform}
    If $\cM$ is compact and  $f : \cM \rightarrow \R^{k+1}$ is Lagrangian dual sectioned, then for any $c \in \R^k$,
    \[
        \Voptf(c) = \max \{y_0 : y_1 = c_1, \dots, y_k = c_k, y \in \conv(f(\cM))\}.
    \]
\end{corollary}

The concavity for $\Voptf$ enables us to use the tools of convex optimization to approach the solution of \eqref{eq:constrained}.
Depending on the particular application, it can be possible to reformulate the optimization problem \eqref{eq:constrained} as a semidefinite program or to use methods such as the ellipsoid algorithm.  We illustrate these possibilities in the examples given throughout the introduction.

Our broader aim is to highlight deeper connections between topology and convexity. For example, our \cref{thm:connected_fibers} focuses on the case $k=1$, and yields the same conclusion as 
\Cref{thm:hidden_cvx} under a topological condition on $\cB$ that is simpler to check in some cases than the existence of a Lagrangian dual section. We also note that our results on convex reformulations for linear inverse spectral problems yield connections to Radon-Hurwitz numbers, which appear prominently in algebraic topology.

For the rest of the introduction, we will give examples of fundamental optimization problems in \Cref{sec:matrix_optimization,subsec:UPP}, and we will preview our results on conditions under which instances of these problems correspond to sectioned functions.
In \Cref{subsec:algorithms}, we will also describe algorithms which can be used to solve sectioned problems, both in their constrained and Lagrangian formulations.

\subsection{Linear Inverse Spectral Problems}
\label{sec:matrix_optimization}
We outline two common classes of optimization problems involving matrices: \emph{quadratically constrained quadratic programs} and \emph{linear Stiefel manifold optimization problems}. Both are known to be NP-hard in general \cite{song2024linear,ben2001lectures}. We view both of these problems as special cases of more general `linear inverse spectral problems', which involve optimization over matrices with fixed eigenvalues or singular values. Indeed, it is possible to consider these problems in an even broader context as optimization problems over the orbits of a Lie group action. For the sake of keeping our exposition elementary, we only outline the connections to Lie theory in \Cref{subsection:lie}.

We give a criterion for the tightness of natural convex relaxations of these problems using the notion of \emph{noncrossing} subspaces, which we define later in this section. 
We defer the proofs, which use the notion of Lagrangian dual sections, to  \Cref{sec:noncrossing}. 

%\subsection{Quadratically constrained quadratic programs (QCQPs) and linear inverse eigenvalue problems}
\subsubsection{QCQPs and Linear Inverse Eigenvalue Problems}\label{subsec:QCQP}

Quadratically constrained quadratic programs (QCQPs) are found in a wide range of applications including engineering \cite{kocuk2016strong}, statistics \cite{candes2015phase}, and combinatorics \cite{gouveia2010theta}.
They are formulated as follows:
For $A_0, \dots, A_k \in \R^{n\times n}_{sym}$, consider
\begin{equation}
    \max \{x^{\intercal} A_0 x : x^{\intercal} A_1 x = c_1, \;\dots,\; x^{\intercal} A_k x = c_k, x \in S^{n-1}\}\label{eq:qcqp}.
    \tag{opt-QCQP}
\end{equation}
This is a special case of \eqref{eq:constrained} over the space $S^{n-1} = \{x \in \R^{n} : x^{\intercal}x = 1\}$, where the function $f$ is $f(x) = (x^{\intercal}A_0x,\dots, x^{\intercal}A_kx)$.

% For $A_0, \dots, A_k \in \R^{n\times n}_{sym}$, consider
% \begin{equation}
%     \max \{x^{\intercal} A_0 x : x^{\intercal} A_1 x = c_1, \;\dots,\; x^{\intercal} A_k x = c_k, x \in S^{n-1}\},\label{eq:qcqp}.
% \end{equation}
% This is a special case of \eqref{eq:constrained}
% where the domain is $S^{n-1} = \{x \in \R^{n} : x^{\intercal}x = 1\}$, and the function $f$ is $f(x) = (x^{\intercal}A_0x,\dots, x^{\intercal}A_kx)$.
% QCQPs have been used to study a wide range of applications ranging from engineering \cite{kocuk2016strong} to statistics \cite{candes2015phase} to combinatorics \cite{gouveia2010theta}.

A QCQP can also be expressed in terms of rank one matrices: \eqref{eq:qcqp} is equivalent to the optimization problem
\begin{equation}
    \max \{\langle A_0, X\rangle : \langle A_1, X\rangle = c_1, \dots, \langle A_k, X\rangle = c_k, X \in \cM_{(1,0,\dots, 0)}\},
\end{equation}
where $\cM_{(1,0,\dots, 0)} = \{xx^{\intercal} : x \in S^{n-1}\} \subseteq \R^{n\times n}_{sym}$.
The notation $\cM_{(1,0,\dots, 0)}$ is meant to suggest an alternative characterization of this set as the set of matrices whose eigenvalues are $1$ with multiplicity~1 and 0 with multiplicity $n-1$. This motivates us to consider the more general domain $\cM_{\lambda} \subseteq \R^{n \times n}_{sym}$, which is the set of $n \times n$ symmetric matrices with prescribed eigenvalues $\lambda_1, \dots, \lambda_n$. 
A \demph{linear inverse eigenvalue problem} is a problem of the form 
\begin{equation}
    \max \{\langle A_0, X\rangle : \langle A_1, X\rangle = c_1, \dots, \langle A_k, X\rangle = c_k, X \in \cM_{\lambda}\},
    \label{eq:inverse_eigenvalue}\tag{opt-LIEP}
\end{equation}
where $\lambda$  is
fixed in advance. Such problems are studied in their own right \cite{candogan2021note, landau1994inverse}.

A standard tool for solving these problems is that of taking a convex relaxation, as these produce upper bounds on the optimal value of the problem.
The natural convex relaxation of \eqref{eq:inverse_eigenvalue} is given by
\begin{equation}
    \max \{\langle A_0, X\rangle : \langle A_1, X\rangle = c_1, \dots, \langle A_k, X\rangle = c_k, X \in \conv(\cM_{\lambda}\}),
    \label{eq:inverse_eigenvalue_conv}\tag{cr-LIEP}
\end{equation}
where $\conv(\cM_{\lambda})$ is known as the \demph{Schur-Horn orbitope}, which has an efficient spectrahedral representation \cite{sanyal2011orbitopes}. For example, \[\conv(\cM_{(1,0,\dots, 0)}) = \{X \in \R^{n \times n}_{sym} : \tr(X) = 1, X \succeq 0\}.\]
It is of considerable interest to understand when this convex relaxation is tight,
We will give a criterion for the convex relaxation to recover the true optimal value of the problem using the notion of a \emph{noncrossing subspace}, which was originally studied in \cite{crossing}.
To define a noncrossing subspace, first recall that an eigenvalue of a symmetric matrix is degenerate if its corresponding eigenspace has dimension greater than one, and otherwise we say that the eigenvalue is nondegenerate.

\begin{definition}\label{def:noncrossing-subspace}
A linear subspace $\subspace \subseteq \R_{sym}^{n\times n}$ is \demph{noncrossing} if every nonzero matrix in $\subspace$ has only nondegenerate eigenvalues.
More generally, we will say a linear subspace $\subspace \subseteq \R_{sym}^{n\times n}$ is \demph{$\ell$-weakly noncrossing} if every nonzero matrix in $\subspace$ has the property that its $\ell$ largest eigenvalues are all nondegenerate.
\end{definition}

Now we state our result for linear inverse eigenvalue problems.
\begin{theorem}
\label{thm:eigenvalues}
    Suppose that $A_0, \dots, A_k\in\R_{sym}^{n\times n}$ span a $\ell$-weakly noncrossing subspace, and $\lambda$ satisfies the condition that $\lambda_{i} = \lambda_{\ell+1}$ for all $i > \ell$.
    Then for any $c \in \R^{k}$, the values of \eqref{eq:inverse_eigenvalue_conv} and \eqref{eq:inverse_eigenvalue} agree, i.e. the convex relaxation of the linear inverse eigenvalue problem is tight.
\end{theorem}
We prove this in \Cref{subsec:liep}.
This result is interesting to a practitioner because it provides a new toolkit for showing that convex relaxations of structured optimization problems give a tight answer. 
It also recovers 
known results when $k=1$
\cite{au1984some}
 and $\lambda=(1,0,\dots,0)$ (\cite{gutkin2004convexity}).

\subsubsection{Linear Inverse Singular Value~Problems and Stiefel Manifold~Optimization}\label{subsubsec:lisv}
Analogously to the previous subsection on linear inverse eigenvalue problems, we will consider the inverse singular value problem. For $\sigma \in \R^{m}_{\ge 0}$ and $n \ge m$, we will let $\Msing$ be the space of $n\times m$ real matrices whose $m$ singular values are the entries of $\sigma$. If $X \in \R^{n\times m}$ is not full rank, then we consider 0 to be a singular value of $X$.

A linear inverse singular value problem given some 
$A_0, \dots, A_k \in \R^{n \times m}$ is defined as
\begin{equation}
    \max \{\langle A_0, X\rangle : \langle A_1, X\rangle = c_1, \dots, \langle A_k, X\rangle = c_k, X \in \Msing\}.
    \label{eq:inverse_singvalue}\tag{opt-LISV}
\end{equation}

%This problem contains many interesting problems as special cases. 
Many important problems are special cases of this one.
For example, if $\sigma = (1,\dots, 1) \in \R^m$ is the all one's vector, then $\Msing$ is the Stiefel manifold of matrices $X \in \R^{n\times m}$ satisfying $X^{\intercal}X = I$. 
For this special case, we will use the notation
\[\St^{n,m} = \cM_{(1,\dots,1)}^{(n,m)} = \{X \in \R^{n\times m} : X^{\intercal}X = I\}.\]
Optimization over $\St^{n,m}$ is used extensively in robotics \cite{peretroukhin2020smooth,dellaert2020shonan, rosen2019se} and chemistry~\cite{bendory2020single}.

The linear inverse singular value problem also includes bilinear optimization, i.e. optimization problems of the form
\begin{equation}
    \max \{x^{\intercal} A_0 y : x^{\intercal} A_1 y = c_1, \dots, x^{\intercal} A_k y = c_k, x \in S^{n-1}, y \in S^{m-1}\},
\end{equation}
as the special case where $\sigma = (1,0,\dots,0)$.

Once again we can consider the convex relaxation,
\begin{equation}
    \max \{\langle A_0, X\rangle : \langle A_1, X\rangle = c_1, \dots, \langle A_k, X\rangle = c_k, X \in \conv(\Msing)\},
    \label{eq:inverse_singvalue_conv}\tag{cr-LISV}
\end{equation}
and ask when it is tight. The set $\conv(\Msing)$ is the \emph{Fan orbitope} in the language of \cite[Section 4.3]{sanyal2011orbitopes}, and has a spectrahedral representation.

In line with the previous section, we %have the 
define some notions of singularly noncrossing subspaces.
For $A \in \R^{n \times m}$ with $m \le n$,
we say that a singular value $\sigma \ge 0$ of $A$ is degenerate if $\sigma^2$ is a degenerate eigenvalue of $A^{\intercal}A$.

\begin{definition}\label{def:singular-noncrossing-subspace}
We say a linear subspace $\subspace \subseteq \R^{n\times m}$ is \demph{$\ell$-weakly singularly noncrossing} if the $\ell$ largest singular values of every nonzero matrix in $\subspace$ are nondegenerate.
\end{definition}

\begin{theorem}
    \label{thm:singvalues}

    Let $m \le n$, and suppose that $A_0, \dots, A_k$ span a $\ell$-weakly singularly noncrossing subspace of $\R^{n \times m}$ for $\ell < m$.
    If either of the following two conditions hold, then \eqref{eq:inverse_singvalue} and \eqref{eq:inverse_singvalue_conv} agree, i.e. the convex relaxation for the inverse singular value problem is tight.

    \begin{itemize}
    \item If $\sigma \in \R_{\ge 0}^m$ is a vector with entries in descending order so that $\sigma_i = 0$ for $i > \ell$.

    \item If $\text{span} \{A_0, \dots, A_k\}$ in addition to being $\ell$-weakly singularly noncrossing contains no nonzero matrix of rank less than $m$, and $\sigma \in \R_{\ge 0}^m$ is a vector with entries in descending order so that $\sigma_i = \sigma_{\ell+1}$ for $i > \ell$.
    \end{itemize}
\end{theorem}
We prove this in \Cref{subsec:lisv}.
In \cite{song2024linear}, it was shown that if $k \le n - m$, and if $\sigma = (1,\dots,1)$, then for any $A_0,\ldots, A_{k} \in \R^{n\times m}$ and for $f = (\langle A_0, X\rangle, \dots, \langle A_k, X\rangle)$, the convex relaxation is tight. This result can be seen to follow from \cref{thm:singvalues} in the case in which $\sigma = (1,\dots, 1)$, as we note in \cref{rmk:steifel}. In \cite[Theorem 11.5]{li2000numerical}, the result when $k=1$ was also shown.

The applicability of our results depends on how commonplace (singularly) noncrossing subspaces are. There are many interesting examples, as we will discuss in the following remark.
\begin{remark}[Comments on the generality of our results]
Whether or not there exists a (singularly) noncrossing subspace of dimension $k+1$ in $\R^{n\times n}_{sym}$ is an interesting question characterized in \cite{crossing}. If $k = 1$, then the Von Neumann-Wigner crossing rule \cite{von1993verhalten} states that almost all (i.e. all but a set of measure 0) pairs $A_0, A_1 \in \R^{n \times n}_{sym}$ span a noncrossing subspace. If $n$ is a power of $2$, then noncrossing subspaces of dimension $k$ only exist when $k \le \log_2(n)$.  For other values of $n$, the largest dimension of a noncrossing subspace of $\R^{n\times n}_{sym}$ dimension is not monotonic in $n$, and is given by the so-called Radon-Hurwitz number.
On the other hand, there exists a $k$-dimensional 1-weakly noncrossing subspace in $\R^{n \times n}_{sym}$ if and only if $k < n$ \cite{friedland1976subspaces}, and this is the case which is relevant for QCQPs.
Moreover, the set of $d$-dimensional $k$-weakly noncrossing subspaces of $\R^{n\times n}_{sym}$ can be thought of as an open subset of an appropriate Grassmannian, and so this phenomenon is not restricted to a set of measure zero.
\end{remark}
\begin{remark}[Connection to Kostant convexity theorem]
Using Kostant's convexity theorem \cite{kostant1973convexity} we can generalize our results to domains $\cM$ which arise as the orbit of a single point under the action of a polar Lie group orbit; see \Cref{subsection:lie}.
This has applications to matrix optimization problems over domains such as the special orthogonal group or settings involving complex numbers, and it is possible to generalize our results using appropriately generalized notions of noncrossing subspaces. We will not write out the analysis in these settings here for the sake of brevity, though convex reparametrizations in the context of the special orthogonal group was analyzed in \cite{ramachandran2024hidden}.
\end{remark}

\subsection{Unbalanced Orthogonal Procrustes Problems}
\label{subsec:UPP}

The Unbalanced Procrustes Problem (UPP) is well-studied~\cite{zhang2020eigenvalue,UPP-relaxation},
and has resisted attempts to understand its computational complexity. In particular, it not known whether this problem is NP-hard or if it admits a polynomial time algorithm.
For $n > m$ and given a pair of matrices, $U \in \R^{n \times d}$ and $W \in \R^{m \times d}$, we have the following instance of the UPP:
\begin{equation}
\label{eq:UPP}\tag{opt-UPP-orig}
    \argmin \left\{
    \|U^{\intercal} X - W^{\intercal}\|^2
    : X \in \St^{n,m} \right\}.
\end{equation}
Here $\|\cdot\|$ denotes the Frobenius norm on matrices. 

This problem has a  geometric description:
if we interpret the columns of $U$ as being points in $\R^n$, then \eqref{eq:UPP} asks us to find a rotation and projection of these points so that their images are as close as possible to the points corresponding to the columns of $W$. 

\begin{figure}
    \centering
    \includegraphics[width=\linewidth]{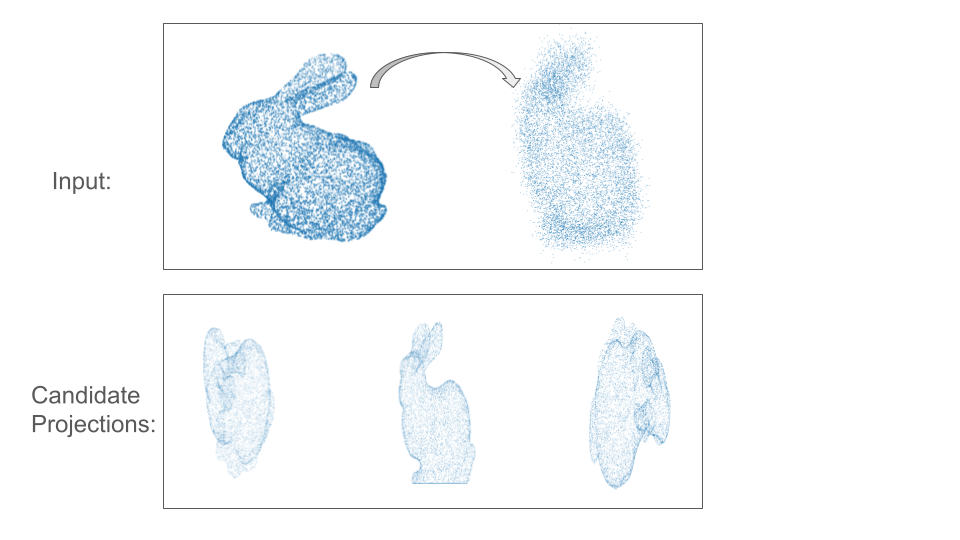}
    \caption{An example of the UPP with $(n,m) = (3,2)$: the columns of $U$ can be thought of as elements of a point cloud in 3D, and the columns of $V$ can be thought of as elements of a 2D point cloud. The goal of the UPP is to find an orthogonal projection of the 3D point cloud that best matches the 2D point cloud. We discuss this example further in \cref{fig:CHORDExample}.
    }
    \label{fig:placeholder}
\end{figure}
We reparametrize this problem in terms of two matrices $A$ and $B$, where $A = -UU^{\intercal}$ and $B = 2UW^{\intercal}$. \Cref{lem:UPP_equivalent} shows that the UPP is equivalent to
\begin{equation}\tag{opt-UPP}
    \argmax_{X\in \R^{n\times m}} \langle B, X\rangle + \langle A, XX^{\intercal}\rangle.
\end{equation}

This reformulation highlights the fact that we can view the objective as being an instantiation of the Lagrangian associated to the function
\[
f : \St^{n,m} \rightarrow \R^2,\; f(X) = (\langle B, X\rangle,\langle A, XX^{\intercal}\rangle).
\]
In particular, the optimal value of the UPP problem is $\displaystyle\max_{X \in \St^{n,m}} f_0(X) + f_1(X)$.

When $m=1$ this problem is equivalent to the \emph{generalized trust region subproblem} of minimizing an inhomogeneous quadratic polynomial on the sphere, which has a tight convex reformulation. We will discuss how we can view this convex reformulation as a consequence of our results in \Cref{ex:UPP1}, in particular showing that $f$ is always sectioned when $m=1.$

It is not the case that all instances of the UPP are sectioned, but we will give a criterion in \Cref{thm:UPP32} for which the UPP with $(n,m) = (3,2)$ is sectioned. Numerical experiments indicate that if $A \in \R^{3\times 3}_{sym}$ and $B \in \R^{3\times 2}$ are sampled uniformly at random from the set of matrices with unit Frobenius norms, then $95\%$ of these instances correspond to cases in which $f$ is sectioned.

\subsection{Algorithms}
\label{subsec:algorithms}
Our results illustrate that the existence of Lagrangian dual sections has strong implications for the geometry of optimization problems such as \eqref{eq:constrained}. Assuming that there exists a Lagrangian dual section, it is natural to ask two algorithmic questions: (1) how can we solve problems of the form $\max_{x \in \cM} \langle t, f(x)\rangle$ for $t \in \tDomain$, and subsequently (2) how can we solve the constrained optimization problems \eqref{eq:constrained}, assuming that we can solve (1) efficiently?

When a problem admits a Lagrangian dual section $D,$ question (1) above asks for a method to compute $D(t)$ for any given value of $t$. One natural approach leverages the continuity of $D$: start with the value of $D(t^{(0)})$ for some $t^{(0)} \in \tDomain$, then track the value of $D(t)$ as $t$ varies from $t^{(0)}$ to its target value. Similarly, there is a natural approach to question (2) about \eqref{eq:constrained}: leverage the convexity of $V_f$ to apply the ellipsoid algorithm from convex optimization.
While these are natural approaches to problems (1) and (2) enumerated above, we will need additional structure to rigorously analyze concrete instantiations of these approaches. This is because there is no way to quantify distances between points in a general topological space $\cM$, and we will need quantitative bounds on distances to control the error incurred by our algorithms. To this end, for our algorithmic results, we will assume that $\cM$ is in fact a Riemannian manifold, and that we can bound the first and second derivatives Lagrangian $\cL(t,x)$. It turns out that this is enough structure for us to give algorithms to solve (1) and (2) with rigorous mathematical guarantees.

In \Cref{sec:alg-in-Riemannian-case}, we leverage the theory of Riemannian manifolds and manifold optimization, and combine this with the notion of Lagrangian dual sections to address problems (1) and (2) above. Our main contribution is a `path-tracking' algorithm for solving problems of the form $\max_{x \in \cM} \langle t, f(x)\rangle$, which we call CHORD (Continuous Homotopy Optimization with Riemannian Descent).
This algorithm combines elements of existing path-tracking / homotopy continuation methods with Riemannian descent techniques, but differs from them in the following important respect---in stark contrast to current methods which either lack convergence guarantees or track paths for all first-order critical points, we can recover the global optimum of a sectioned problem by tracking a single path.  In particular, while our methods ensure a global optimum for a sectioned problem, Riemannian gradient descent (RGD) on its own does not come with such guarantees and it may become trapped in local optima.  We give an example for the UPP described above in which CHORD returns a near-global optimizer, while RGD does not in \Cref{ex:upp-bunny}.  Our analysis combines the continuity properties of sectioned problems and convergence theory for RGD.  

\subsection{Related Work}
There has been an extensive body of work concerning hidden convexity in the sense of finding convex images of sets under projection. Perhaps the most classical example is the Schur-Horn theorem \cite{schur1914zwei, horn1954doubly}, and its generalization due to Kostant \cite{kostant1973convexity}. There is also the classical Toeplitz-Hausdorff theorem concerning different `generalized numerical ranges' \cite{toeplitz1918algebraische, hausdorff1919wertvorrat, davis1971toeplitz}. More recent works along these lines on projections of sets under linear maps include \cite{brickman1961field,binding1985hermitian,gutkin2004convexity, au1983extension, au1984some, li2000numerical, song2024linear}. Of these, the most similar in spirit to our work is \cite{gutkin2004convexity}, which applied differential geometry to consider the convexity of the image of a collection of quadratic forms on a sphere, and similarly analyzed the Lagrangian. Our work unifies the proofs of many of the results in these earlier papers, in addition to having more implications beyond those covered in the earlier literature.

There is also a large literature on the exactness of semidefinite programming relaxations. In particular, the S-lemma is a standard result in optimization, which is equivalent to Brickman's theorem that the image of a sphere under two quadratic maps is convex. Much of this work considers different constraint qualification conditions, such as those of `linear independence type'. For example, \cite{cifuentes2022local} discusses the stability of the exactness of semidefinite programming relaxations under perturbations. Our work does not directly consider such constraint qualification conditions, though it may be interesting to consider whether it is possible to unify these points of view on semidefinite programming tightness.

There is also an extensive literature on manifold optimization, which directly minimizes a function on a manifold without passing to a convex relaxation. To highlight a few examples in the context of Grassmannian optimization, we refer the reader to~\cite{EGNS2021-gro,LLY2025-gro,YWL2022-gro,LWY2021-gro}. For a general perspective on such problems, we  recommend the textbook \cite{boumal}.  Our theory is distinct from this prior manifold optimization literature in that we allow for additional nonlinear constraints to manifold optimization problems (these additional constraints cannot be decoupled from the manifold constraint).  In contrast, the earlier theory typically only considers function minimization on the manifold.

The methods of~\Cref{sec:alg-in-Riemannian-case} fuse Riemannian optimization with homotopy continuation methods. The latter are well-known in optimization (see e.g.~\cite[\S 11.3]{nocedal2006numerical}.) 
They are also used prominently in algebraic rootfinding to track multiple solution paths over the complex numbers; this ensures all roots are found `with probability one' (see e.g.~\cite[Ch.~4 \& 7]{sommese-wampler} for precise statements.)
When solving first-order optimality equations, this approach may not scale due to large numbers of complex roots (see e.g.~\cite{dhost,nie-ranestad}).
In stark contrast, we track just one path over the reals to find the global optimum of any sectioned problem.

Some of our results concern subspaces of matrices with various spectral properties. These have been studied extensively in algebraic topology.  For example \cite{adams1965matrices} relates the question of the dimensions of subspaces of $\R^{n\times m}$ which include no singular matrices to Adams' theory of linearly independent vector bundles on the sphere. These results were extended to `noncrossing subspaces' in \cite{crossing}.

\subsection{Notation} 

We let $\R_{sym}^{n\times n}$ denote the vector space of $n\times n$ symmetric matrices. If $A, B \in \R^{n\times m}$, we recall that the trace inner product is defined by $\langle A, B\rangle = \tr(A^{\intercal}B)$. For a set $S \subseteq \R^k$, we let $\conv(S)$ denote the convex hull of $S$, defined as the intersection of all convex sets containing $S$.

\section{Convexity and Topology}\label{sec:convexity-topology}

The main purpose of this section is to lay out some fundamental connections between elementary notions from topology and convexity.
In particular, we will emphasize the topological properties of the Lagrangian dual bundle and how they imply the convexity of the optimal value function $V_f(c)$ defined by \eqref{eq:constrained}.  The section is organized as follows.  In \Cref{subsec:main_proof} we prove our main result \Cref{thm:hidden_cvx}, which shows that the existence of a Lagrangian dual section implies the existence of a natural convex reformulation for the problem \eqref{eq:constrained}.  In \Cref{sec:unique_maximization}, we show that if the maximizer of the Lagrangian $\mathcal{L}(t,x)$ is unique for each $t \in \tDomain$, then $f$ is sectioned.  We also show that this property of having unique maximizers is preserved under a natural operation in \Cref{lem:concave_composition}.
Finally, in \Cref{subsec:connected_fibers} the connectedness of the fibers of the Lagrangian bundle are also shown to imply the existence of a natural convex reformulation. Here the fibers of the Lagrangian bundle are the sets of the form $\cM_t = \{(t,x) : x \in \argmax \langle t, f(x)\rangle\}$.

\subsection{Proof of \Cref{thm:hidden_cvx}}
\label{subsec:main_proof}
In finite dimensions, the convex hull of a compact set is compact \cite[Corollary 2.4]{barvinok}, and in particular, this implies that if $\cM$ is compact and $f : \cM\rightarrow \R^{k+1}$ is continuous, then $\conv f(\cM)$ is compact. We may also assume that $\conv f(\cM)$ is full-dimensional, i.e., it contains an interior point, for otherwise we could just restrict $f$ to be a map from $\cM$ to an affine subspace of $\R^{k+1}$.
\newcommand{\fminuszero}{f_{\operatorname{CM}}}
In this section, we use the notation $\fminuszero(x)$ for  $(f_1(x),\dots,f_k(x))$, where $\operatorname{CM}$ is meant to abbreviate `constraint map'.

%We will show our main theorem by proving 
We aim to prove that the Lagrangian dual of \eqref{eq:constrained}
satisfies a version of \emph{strong duality}. To be precise, we define the function
\[
    \widehat{V}_f(c) = \inf_{\lambda \in \R^k} \sup_{x \in \cM} \mathcal{L}\left(\begin{pmatrix} 1\\ \lambda\end{pmatrix}, x\right) - \langle c, \lambda\rangle,
\]
%We will then show for all $c \in \R^{k}$, 
and show that for all $c \in \R^{k}$,
\begin{equation}\label{eq:dual-equals-primal}
    \widehat{V}_f(c) = V_f(c).
\end{equation}
This function $\widehat{V}_f(c)$ is
analogous to the double Fenchel conjugate of $V_f$, and %in particular, $\hat{V}_f(c)$ is clearly concave since it is the pointwise minimum of linear functions in $c$. 
since it is the pointwise minimum of linear functions in $c$, it is concave.
Thus, %if we show 
proving
the equality 
\eqref{eq:dual-equals-primal} implies that $V_f(c)$ is concave.

At a high level, we will prove $\widehat{V}_f(c) = V_f(c)$ by showing that 
$\Voptf(c)$ is bounded above and below by $\widehat{V}_f(c)$.
The bound $\widehat{V}_f(c) \ge \Voptf(c)$ is a standard consequence of weak duality.
Proving the second bound, $\Voptf(c)\ge \widehat{V}_f(c)$, involves using the Lagrangian dual section to show the existence of some $\lambda \in \R^k$ for which the inner maximization gives a value of at most $\widehat{V}_f(c)$. In order to show the existence of such a $\lambda$, we will employ some ideas from algebraic topology.

 \subsubsection*{A lower bound on $\widehat{V}_f(c)$.}

 We may assume that $\Voptf(c) \neq -\infty$. 
 By compactness of $\cM$ there exists a point $x_c\in \cM$ such that
 \[
 f(x_c) = \begin{pmatrix} V_f(c) \\ c \end{pmatrix}.
 \]

For any $\lambda \in \R^k$, we have
\begin{align*}
\mathcal{L}\left(
    \begin{pmatrix} 1\\ \lambda\end{pmatrix}, x_c
    \right) 
    - \langle c, \lambda\rangle 
    \;=\;& f_0(x_c) + 
    \langle \fminuszero(x_c), \lambda\rangle 
    - 
    \langle c, \lambda\rangle \\
    \;=\;& V_f(c).
\end{align*}
Since this holds for each $\lambda$, the infimum over $\lambda\in\mathbb{R}^k$ satisfies the upper bound:
\[\widehat{V}_f(c)=
\inf_{\lambda \in \R^k} \sup_{x \in \cM} \mathcal{L}\left(\begin{pmatrix} 1\\ \lambda\end{pmatrix}, x\right) - \langle c, \lambda\rangle
\ge V_f(c).
\]

\subsubsection*{Towards an upper bound on $\widehat{V}_f(c)$.}

Let $D(t)$ be the Lagrangian dual section for $f$ that is guaranteed by the hypotheses of \Cref{thm:hidden_cvx}.
We want to show that for any $c\in\mathbb{R}^k$ for which $\widehat{V}_f(c) \neq -\infty$ we have $V_f(c) \ge \widehat{V}_f(c)$. The following lemma reduces this problem to that of finding some $t \in \tDomain$ so that $\fminuszero(D(t)) = c$.

\begin{lemma}
    
    If there exists $t\in\tDomain$ such that $\fminuszero(D(t))=c$,
    then $V_f(c) \ge \widehat{V}_f(c)$.
\end{lemma}

\begin{proof}
If such a $t$ exists, then we are in one of two cases.
In the first case, $t_0 > 0$. In this situation, we let 
$\lambda = \frac{1}{t_0}{(t_1,\dots,t_k)}$ 
and get the desired inequality:
\begin{align*}
    \widehat{V}_f(c) &\le \sup_{x \in \cM} \mathcal{L}\left(\begin{pmatrix} 1\\ \lambda\end{pmatrix}, x\right) - \langle c, \lambda\rangle\\
    &= \mathcal{L}\left(\begin{pmatrix} 1\\ \lambda\end{pmatrix}, D(t)\right) - \langle c, \lambda\rangle\\
    &= f_0(D(t)) + \langle \fminuszero(D(t)), \lambda\rangle - \langle c, \lambda\rangle\\
    &\le V_f(c).
\end{align*}
For the case $t_0 = 0$, we can apply a sequence of the same inequalities with
$\lambda^{(j)} = j\cdot (t_1,\dots,t_k)$ for each $j \in \N$.
Letting $j \to \infty$ gives the desired lower bound on $\Voptf(c)$.
\end{proof}

\subsubsection*{Proving surjectivity.}
%Thus, the problem
Completing the proof
reduces to showing that for any $c$ such that $\widehat{V}_f(c) \neq -\infty$, there exists some $t \in \tDomain$ 
%so that
with
$\fminuszero(D(t)) = c$.
The separating hyperplane theorem implies that
\[\{c \in \R^k : \widehat{V}_f(c) \neq - \infty\} = \conv \fminuszero(\cM).\]
Letting $\cC = \conv \fminuszero(\cM)$, the main theorem follows from the next result: 

\begin{proposition}
\label{prop:surjective}
    Let $H_k = \{t \in \R^{k+1}_{0+} : \|t\|=1\}$. The map $\phi : H_k \rightarrow \cC$ defined by $\phi(t) = \fminuszero(D(t))$ is surjective.
\end{proposition}
   
The remainder of this section is devoted to proving \Cref{prop:surjective}. We will be using some topological notions in order to accomplish this.
Note that $H_k$ and $\cC$ are each homeomorphic to the closed ball of dimension $k$, $B^k = \{x \in \R^k : \|x\| \le 1\}$, so the map $\phi$ can be regarded, up to homeomorphism, as a continuous map from the unit ball to itself. This is the setting of \cref{lem:homotopy_fact} from algebraic topology, which will be the way we ultimately show the surjectivity of $\phi$.

Before stating \cref{lem:homotopy_fact}, we will need to recall the notion of a homotopy equivalence, which is also defined in \cite[Chapter 0]{MR1867354}.
If $\cM_1$ and $\cM_2$ are topological spaces, and $g_1, g_2 : \cM_1 \rightarrow \cM_2$ are continuous maps, then a \emph{homotopy} from $g_1$ to $g_2$ is given by a continuous map $H : \cM_1 \times [0,1] \rightarrow \cM_2$ such that $H(x, 0) = g_1(x)$ and $H(x,1) = g_2(x)$ for all $x \in \cM_1$. A \emph{homotopy inverse} for a map $g : \cM_1 \rightarrow \cM_2$ is a map $h : \cM_2 \rightarrow \cM_1$ so that there is a homotopy from the composition $g \circ h$ to the identity on $\cM_2$ and there is a homotopy from the composition $h \circ g$ to the identity on $\cM_1$. Finally, we say that $g : \cM_1 \rightarrow \cM_2$ is a \emph{homotopy equivalence} if there exists a homotopy inverse for $g$.

\begin{lemma}
\label{lem:homotopy_fact}
    Suppose that $\phi : B^k \rightarrow B^k$ is a continuous map. Let $\partial \, \phi$ be the restriction of $\phi$ to the boundary $\partial B^k$ and suppose that the image of $\partial{\phi}$ is contained in $\partial B^k$. If $\partial{\phi} : \partial B^k \rightarrow \partial B^k$ is a homotopy equivalence, then $\phi$ is surjective.
\end{lemma}

The proof of \Cref{lem:homotopy_fact} is very similar to the proof to the proof of the Brouwer fixed point theorem, and we include it for completeness in \Cref{app:proofs}.

We apply \cref{lem:homotopy_fact} to $\phi = \fminuszero \circ D$ to show \cref{prop:surjective}. We first need to show that the map $\partial{\phi} : S^{k-1} \rightarrow \cC$ maps onto the boundary of $\cC$ (where $S^{k-1} = \partial H_k = \{t \in \R^{k+1}: t_0 = 0, \|t\|=1\}$). This is clear because if $t \in S^{k-1}$, then $D(t)$ maximizes the linear function
\[\langle t, f(x) \rangle = \langle (t_1,\dots,t_k), \fminuszero(x)\rangle
\]
over $x \in \cM$, and thus $\fminuszero(D(t))$ must be on the boundary of $\cC$.

It remains to show that $\partial{\phi}$ is a homotopy equivalence between $S^{k-1}$ and $\partial \cC$. We do this in the following lemma:
\begin{lemma}
    Suppose that $\cC \subseteq \R^k$ is a full dimensional convex set, and that $\partial{\phi}: S^{k-1} \rightarrow \partial \cC$ satisfies the property that for any $t \in S^{k-1}$, 
    \[
    \partial{\phi}(t) \in \argmax_{y\in\cC} \langle t, y\rangle.
    \] 
    Then $\partial{\phi}$ is a homotopy equivalence.
\end{lemma}
\begin{proof}
    Fix some $y_0$ in the interior of $\cC$, and let $\psi : \partial \cC \rightarrow S^{k-1}$ be defined by $\psi(y) = \frac{y-y_0}{\|y-y_0\|}$. This map $\psi$ is clearly continuous, and standard results from convex geometry show that $\psi$ is a bijection. Since $S^{k-1}$ is a compact Hausdorff space, this implies that $\psi$ is a homeomorphism. We claim that $\psi$ is a homotopy inverse to $\partial {\phi}$.

    To see this, we first define the map 
    \[
        H: S^{k-1} \times [0,1] \rightarrow S^{k-1},
        \quad 
        (t,s) \mapsto \frac{s t + (1-s)\psi(\partial{\phi}(t))}{\|s t + (1-s)\psi(\partial{\phi}(t))\|},
    \]
   which projects a convex combination of $t$ and $\psi(\partial{\phi}(t))$ onto the sphere. 
    To show that $H$ is well-defined, we must show $s t + (1-s)\psi(\partial{\phi}(t)) \neq 0$. To see this, note first the equality
    \begin{equation}\label{eq:abs-H-denom}
        \langle t, s t + (1-s)\psi(\partial{\phi}(t)) \rangle  
        \,=\, 
        s \|t\|^2 + (1-s) \frac{\langle t, \partial\phi(t) - y_0\rangle}{\|\partial{\phi}(t) - y_0\|}.
    \end{equation}
    Since $\partial{\phi}(t)$ maximizes $\langle t, y\rangle$ on $\cC$ and $y_0$ is in the interior of $\cC$, we have $\langle t, \partial{\phi}(t) - y_0\rangle > 0$ for all $t \in S^{k-1}$. Thus,~\eqref{eq:abs-H-denom} is positive for all $t, s$, and in particular, $s t + (1-s)\psi(\partial{\phi}(t)) \neq 0$.
    Note that $H(t,0) = \psi(\partial{\phi}(t))$  and $H(t,1) = t$, which shows that $H$ is a homotopy from $\psi \circ \partial \phi$ to the identity map.

    Since $\psi$ is a homeomorphism, we can conjugate $H$ by the map $\psi$ to obtain a homotopy from $\partial{\phi} \circ \psi$ to the identity, i.e. we may set
    \[
        H': \partial \cC \times I \rightarrow  \partial \cC,\;H'(y,s) = \psi^{-1}(H(\psi(y), s)).
    \]
    This is also continuous, and it satisfies the property that 
    \[H'(y,0) = \psi^{-1}(\psi(\partial \phi(\psi(y)))) = \partial{\phi}(\psi(y))\] and $H'(y,1) = y$.
    This shows that $\partial{\phi}$ is a homotopy equivalence.
\end{proof}

\subsection{Unique Maximization and Lagrangian Dual Sections}
\label{sec:unique_maximization}
Here we give a simple criterion for a function to be sectioned.
\begin{definition}
We say that  
$f : \cM \rightarrow \R^{k+1}$ is \emph{uniquely maximized}, if for each $t \in  \Rplus^{k+1}$, 
the function $\ell_t:\cM\to\mathbb{R}$, 
$x\mapsto \langle t, f(x)\rangle$
has a unique maximizer.
\end{definition}

\medskip
\begin{lemma}
\label{lem:unique_max}
    Fix a compact metric space $\cM$ and a continuous function $f : \cM \rightarrow \R^{k+1}$ and its associated Lagrangian $\cL(t,x)$. If $f$ is uniquely maximized, then $f$ is sectioned.
\end{lemma}

\begin{proof}
\label{prf:unique_max}
Let $D(t)$ be the unique maximizer of the function $\langle t, f(x)\rangle$ for each $t \in \Rplus^{k+1}$.
Because $\cM$ is a metric space, it suffices to show that if $\left(t^{(i)}\right)_{i\ge 1} $ is a sequence in $\tDomain $ and $\lim_{i \rightarrow \infty} t^{(i)} = t_*$, then $\lim_{i \rightarrow \infty} D(t^{(i)}) = D(t_*)$.
We may assume that $\{D(t^{(i)})\}_{i=1}^{\infty}$ in fact has a limit $x_*$ by passing to a subsequence and using the compactness of $\cM$. We need to prove that $x_* = D(t_*)$.

Since $D(t_*)$ is the unique maximizer of $\langle t_*, f(x)\rangle$, it suffices to establish the claim that $x_*$ also maximizes this function.
That is, for any $x \in \cM$, we wish to argue that 
\[
    \langle t_*, f(x_*) - f(x) \rangle \ge 0.
\]
Note that 
\[
    \langle t_*, f(x_*)-f(x) \rangle = \lim_{i\rightarrow \infty} \langle t^{(i)}, f(D(t^{(i)})) - f(x) \rangle.
\]
We have that $\langle t^{(i)}, f(D(t^{(i)})) - f(x) \rangle > 0$ for all $i$ by the definition of $D(t^{(i)})$, so that the limit is also nonnegative. This proves the claim. 
\end{proof}

A property of uniquely maximized functions is that they behave well under composition with monotonic increasing functions.

\begin{lemma}
\label{lem:concave_composition}
    Suppose that $\tilde f = (\tilde f_0, \tilde f_1, \dots, \tilde f_{k+1}) : \cM \rightarrow \R^{k+1}$ is uniquely maximized. If $g : \R \rightarrow \R$ is any monotonically increasing strictly concave function, then the function $f(x) = (g(\tilde f_0(x)),\tilde f_1(x) \dots, \tilde f_{k+1}(x))$ is uniquely maximized.
\end{lemma}
\begin{proof}
We have seen that the function
\[
    V(c) = \max \{\tilde f_0(x)  : \tilde f_1(x) = c_1, \dots, \tilde f_k(x) = c_k, x \in \cM\}
\]
is concave because $f$ is sectioned.
If we let
\[
    \widetilde{V}(c) = \max \{g(\tilde f_0(x))  : \tilde f_1(x) = c_1, \dots, \tilde f_k(x) = c_k, x \in \cM\},
\]
then it follows from the monotonicity of $g$ that $\widetilde{V}(c) = g(V(c))$. Since $g$ is both monotonically increasing and strictly concave, the function  $\widetilde{V}(c)$ is also strictly concave.

Fix a $t \in \tDomain$, and consider the maximum of $\langle t, f(x)\rangle$. This is equal to the maximum value of the function
\[
    p(c) = t_0g(V(c)) + \langle (t_1,\dots,t_k), (c_1,\dots,c_k)\rangle,
\]
over $c$.
Because $g(V(c))$ is strictly concave, $p(c)$ is strictly concave, and in particular, $p(c)$ has a unique maximizer, say $c^*$. 

We now want to claim that there is a unique $x \in \cM$ so that $f(x) = (V(c^*), c^*)$. To see this, note that $(V(c^*), c^*)$ is on the upper boundary of $\conv f(\cM)$, and therefore, by the separating hyperplane theorem, there exists a $t' \in \R_{\ge 0} \times \R^k$ such that $\max \langle t', f(x)\rangle = \langle t', (V(c^*), c^*)\rangle$. By unique maximization of $f$, we have therefore that there is a unique $x^*$ maximizing $\langle t', f(x)\rangle$ on $\cM$, and hence there is a unique $x^*$ mapping to $(V(c^*), c^*)$, which is also consequently the unique maximizer of $\langle t, \tilde{f}(x)\rangle$.
\end{proof}

\subsection{Connectedness of the Fibers of the Lagrangian Bundle}
\label{subsec:connected_fibers}
In the special case when there is $k=1$ constraint, there is a simple topological condition on the fibers of the Lagrangian bundle of $f$ which implies the existence of a convex reparametrization, namely, that all of the fibers are connected.

\begin{theorem}
\label{thm:connected_fibers}

    Let $\cM$ be a compact topological space, and $\cB$ be the Lagrangian bundle associated to $f : \cM \rightarrow \R^2$, as defined in~\eqref{eq:bundle}. If for each $t \in \Rplus^{2}$ the fiber 
    $\cM_t = \{x \in \cM : (t,x) \in \cB\}$ is connected, then the function $V_f(c)$ is concave.
\end{theorem}
\begin{proof}
Let
\[
    \widehat{V}_f(c) = \max \{ y_0 : y_1 = c_1, y \in \conv f(\cM)\}.
\]
We will again argue that $\widehat V_f(c) = V_f(c)$ for all $c \in \R$. We note that $\widehat V_f(c) \ge V_f(c)$ for all $c$, since $\conv f(\cM) \supseteq f(\cM)$.
Fix some $c$ so that $\widehat V_f(c) \neq -\infty$, and $\hat{y} \in \conv f(\cM)$ satisfying $\hat{y}_1 = c_1$ and also $\hat{y}_0 = \widehat V_f(c)$.
It is clear that if $\hat{y}$ cannot be in the interior of $\conv f(\cM)$ and still maximize $y_0$ subject to $y_1 = c_1$. It follows that $\hat y$ must be on the boundary of $\cM$, and therefore, it is contained in a face of $\conv f(\cM)$ of dimension at most one. If $\hat y$ is contained in a zero-dimensional face of $\conv f(\cM)$, then it is an extreme point, and hence in $f(\cM)$, and this implies that $V_f(c) \ge \widehat V_f(c)$. 

Assume now that $\hat{y}$ is contained in a one-dimensional face of $\conv f(\cM)$, which must be a line segment $[y^{(1)}, y^{(2)}]$ spanned by two extreme points $y^{(1)}, y^{(2)} \in f(\cM)$.
Thus, there exists some $t \in \R_{\ge 0} \times \R$ so that the linear function $\langle t, y\rangle$ on $\conv f(\cM)$ attains its maximum value $v$ in the interval $[y^{(1)}, y^{(2)}]$. 
Since $\cM_t = \{x \in \cM : \langle t, f(x)\rangle = v \}$ is connected, so is the set $f(\cM_t) \subseteq [y^{(1)}, y^{(2)}].$
Moreover, as $y^{(1)}, y^{(2)} \in \cM_t ,$ we have $f(\cM_t) = [y^{(1)}, y^{(2)}].$
In particular, $\hat{y} \in f(\cM_t)$, so that $V_f(c) = \widehat V_f(c)$.
\end{proof}

\begin{remark}
\Cref{thm:connected_fibers} can be used to give a simple proof of the Schur-Horn theorem for $3\times 3$ symmetric matrices.
It also gives a simple proof of the fact that, for example, if $f : S^{n-1} \rightarrow \R^2$ is a quadratic map $f(x) = (x^{\intercal}A_0x, x^{\intercal}A_1x)$ for symmetric matrices $A_0, A_1$, then $V_f(c)$ is concave. In this case, for any $t \in \R^2$, it is not hard to see that the fiber of the Lagrangian bundle at $t$ is given by 
\[
\{(t, x) : x \in S^{n-1}, x \text{ is an eigenvector of }t_0A_0 + t_1A_1 \text{ with maximum eigenvalue}\},
\]
which is clearly connected. Similar simple proofs also apply in the settings of inverse eigenvalue and singular value problems with one constraint.
\end{remark}

\section{Lagrangian Dual Sections for Linear Inverse Spectral Problems}\label{sec:noncrossing}

This section explains connections between the various notions of noncrossing subspaces and Lagrangian dual sections. These noncrossing subspaces 
are related to representations of Clifford algebras and have been studied in physics \cite{crossing, von1993verhalten}.
We show that these noncrossing subspaces are intimately connected to the exactness of certain SDP for linear inverse spectral problems.  In particular, we prove \Cref{thm:eigenvalues,thm:singvalues} by showing that certain maps associated to the linear inverse eigenvalue and singular value problems are sectioned. We prove these results in \Cref{lem:noncrossing,lem:sing} in \Cref{subsec:liep,subsec:lisv} respectively.

It turns out that we can give a unified proof of both of these results in the context of polar representations of Lie groups using the Kostant convexity theorem (and its variants). We discuss this approach and the common generalization of our results which is afforded by this approach in \Cref{subsection:lie}.  (Despite this unifying generalization, we still retain the exposition in \Cref{subsec:liep,subsec:lisv} focusing on the examples of linear inverse spectral problems as the proofs in those cases only involve well-known facts from linear algebra.)

\subsection{Preliminaries on Convex Reparametrizations of Linear Inverse Spectral Problems}
We discuss some preliminaries pertaining to convex reparametrization of linear spectral problems that we use throughout this section.  We recall from~\Cref{def:noncrossing-subspace} the notion of an ($\ell$-weakly) noncrossing subspace, and from~\Cref{def:singular-noncrossing-subspace} its singular-value analog.
For $X \in \R^{n \times n}_{sym}$, let $\lambda_1(X) \ge \dots \ge \lambda_n(X)$ denote the eigenvalues of $X$ (accounting for multiplicity). Similarly, for $X \in \R^{n \times m}$ with $m \le n$, let $\sigma_1(X) \ge \dots \ge \sigma_m(X)$ denote the singular values of $X$. 
For $\lambda \in \R^n$, we recall
\[
    \cM_{\lambda} = \{X \in \R^{n \times n}_{sym} : \forall i = 1, \dots, n,\; \lambda_i(X) = \lambda_i\}.
\]
Similarly, for $\sigma \in \R^m$,  
\[
    \cM^{(n,m)}_{\sigma} = \{X \in \R^{n \times m} : \forall i = 1, \dots, m,\; \sigma_i(X) = \sigma_i\}.
\]

The following proposition shows that if $f$ is a linear map, then the concavity of $V_f$ is sufficient to show that replacing $\cM$ with $\conv(\cM)$ in the constrained optimization problem \eqref{eq:constrained} gives the same optimal value function. Thus, \cref{thm:eigenvalues} and \cref{thm:singvalues} reduce to showing that $V_f$ is concave under the hypotheses of the respective theorems, which we will ultimately do by showing that $f$ is sectioned.
\newcommand{\barf}{f}
\begin{proposition}
    \label{prop:linear_proj}
    Suppose that $\cM \subseteq \R^N$ is a compact set and that $\barf : \R^N \rightarrow \R^{k+1}$ is a  linear map. If $V_f(c) = \max \{f_0(x) : f_1(x) = c_1, \dots, f_k(x) = c_k, x \in \cM\}$ is concave, then
    \begin{align*}
        V_f(c) = \max \{\barf_0(x) : \barf_1(x) = c_1, \dots, \barf_k(x) = c_k, x \in \conv(\cM)\}.
    \end{align*}
\end{proposition}
\begin{proof}
This follows from \cref{cor:convex_reform} by simply noting that if $f$ is linear, then $\conv f(\cM) = f(\conv(\cM))$.
\end{proof}

With~\Cref{prop:linear_proj} in hand, we may now prove~\Cref{thm:eigenvalues,thm:singvalues} in the next two subsection, with the arguments following a similar structure.

\subsection{Noncrossing Subspaces and Sectioned Linear Inverse Eigenvalue Problems}\label{subsec:liep}

We prove~\Cref{thm:eigenvalues} by using the following lemma: 
\begin{lemma}
    \label{lem:noncrossing}
    Let $A_0, \dots, A_k \in \R^{n\times n}_{sym}$ be linearly independent elements of a $\ell$-weakly noncrossing subspace. For any $\lambda \in \R^n$ satisfying $\lambda_i = \lambda_{\ell+1}$ for all $i \ge \ell+1$, the map $f:\cM_{\lambda} \rightarrow \R^{k+1}$ defined below is sectioned: 
    \[
        f(X) = (\langle A_0, X\rangle, \dots, \langle A_k, X\rangle).
    \]
\end{lemma}
\begin{proof}[Proof of \Cref{lem:noncrossing}]
\newcommand{\At}{A(t)}
Fix any $\lambda \in \R^n$ with $\lambda_{k+1} = \lambda_{k+2} = \dots = \lambda_n$, and an arbitrary $t \in \tDomain $.
By~\Cref{lem:unique_max}, it suffices to show that the function 
\[
    \ell_t(X) 
    = \left\langle\sum_{i=0}^{d} t_iA_i, X\right\rangle
\]
has a unique maximizer in $\cM_{\lambda}$.
Since $\cM_{\lambda}$ is invariant under orthogonal change of basis, we may assume that the matrix $\sum_{i=0}^{d} t_iA_i$ is diagonal, i.e. $\sum_{i=0}^{d} t_iA_i = \Diag(\mu)$ for some $\mu \in \R^n$. Having assumed that $\{A_0, \dots, A_k\}$ span a $\ell$-weakly noncrossing subspace, we may further assume that 
\[
\mu_{1} > \dots > \mu_{\ell} \ge \mu_{\ell+1} \ge \dots \ge \mu_n.
\]
We have
\[
    \ell_t(X) 
    = \langle\Diag(\mu), X\rangle
    = \langle\mu,  \diag(X)\rangle,
\]
where $\diag(X) \in \R^n$ denotes the vector of diagonal entries of the matrix $X$.

By the Schur-Horn theorem, the extreme points of $\{ \diag(X) : X \in \cM_{\lambda}\}$ are vectors of the form $\tau \lambda$, where $\tau $ is some permutation of the indices $1,\ldots , n$ and $\tau \lambda = (\lambda_{\tau^{-1}(i)})_{i=1}^n$. It follows that $\argmax_y \{\langle\mu, y\rangle  : y = \diag(X), X \in \cM_{\lambda}\} = \lambda$, where 
\[
    \lambda_1  \ge \dots \ge \lambda_{\ell} = \lambda_{\ell+1} = \dots = \lambda_n.
\]

In particular, we have that any maximizer of $\ell_t$ over $\cM_{\lambda}$ must satisfy $\diag(X) = \lambda$, which clearly implies that $X = \Diag(\lambda)$ is the unique maximizer of $\ell_t$.
\end{proof}

Combining~\Cref{prop:linear_proj} and~\Cref{lem:noncrossing}, we deduce~\Cref{thm:eigenvalues} below.

\begin{proof}[Proof of \Cref{thm:eigenvalues}]
    Let $A_0, \dots, A_k \in \R^{n \times n}_{sym}$ span a $\ell$-weakly noncrossing subspace, and let $\lambda \in \R^n$ satisfy $\lambda_{\ell+1} = \lambda_i$ for $i \ge \ell+1$.
    We may reduce to the case where $A_0, \dots, A_{k}$ are linearly independent, by linear change of coordinates for both the original and relaxed problems.
    \Cref{lem:noncrossing} then implies that the map $f : \cM_{\lambda} \rightarrow \R^{k+1}$ defined by $f(X) = (\tr(A_0X), \dots, \tr(A_{k}X))$ is sectioned.
    We may therefore conclude, by \Cref{prop:linear_proj}, that 
    \[
    V_f(c) = \max \{\langle A_0, X\rangle : \langle A_1, X\rangle = c_1, \dots, \langle A_{k}, X\rangle = c_{k}, X \in \conv(\cM_{\lambda})\}.
    \]
\end{proof}

\subsection{Singularly Noncrossing Subspaces and Sectioned Linear Inverse Singular Value Problems}
\label{subsec:lisv}
Here is a singular-value analog of~\Cref{lem:noncrossing}:

\begin{lemma}
    \label{lem:sing}
    Let $n \ge m$, and $A_0, \dots, A_k \in \R^{n\times m}_{sym}$ be linearly independent elements of a $\ell$-weakly singularly noncrossing subspace.
    If either of the following two conditions hold, then the map  $f:\cM_{\sigma}^{(n,m)} \rightarrow \R^{k+1}$ below is sectioned:
    \[
        f(X) = (\tr(A_0 X), \dots, \tr(A_kX)).
    \]
    \begin{itemize}
    \item If $\sigma \in \R_{\ge 0}^m$ is a vector with entries in descending order so that $\sigma_i = 0$ for $i \ge \ell$.

    \item If $\text{span} \{A_0, \dots, A_k\}$, in addition to being $\ell$-weakly singularly noncrossing, contains no nonzero matrix of rank less than $m$, and $\sigma \in \R_{\ge 0}^m$ is a vector with entries in descending order so that $\sigma_i = \sigma_{\ell+1}$ for $i > \ell$.
    \end{itemize}
\end{lemma}
\begin{proof}[Proof of \Cref{lem:sing}]
This proof is almost identical to that of \Cref{lem:noncrossing}, where the relevant facts about eigenvalues are replaced by singular values.
We now fix $\sigma \in \R^m$ satisfying the assumptions of the theorem. We want to apply~\Cref{lem:unique_max}, which requires showing that the function 
\[
    \ell_t(X) = \tr\left(\sum_{i=0}^{d} t_iA_i X\right)
\]
has a unique maximizer over the set $\cM_{\sigma}^{(n,m)}$ for any $t \in \tDomain$.
Since $M_{\sigma}^{(n,m)}$ is invariant under the bi-orthogonal action 
\[
(U,V) \cdot X = U X V^T 
\quad (U,V)\in O(n) \times O(m),
\, \, X\in \cM_{\mu}^{(n,m)},
\]
we may assume $\sum_{i=0}^{d} t_iA_i = \Diag(\mu)$ for some $\mu \in \R^m$.
Because $\{A_0, \dots, A_k\}$ span a $\ell$-weakly singularly noncrossing subspace, we may assume that
\[
    \mu_1 > \dots > \mu_{\ell} > \mu_{\ell+1} \ge \dots \mu_m \ge 0.
\]

As in the previous proof, we have that 
\[
    \ell_t(X) = \langle \Diag(\mu), X\rangle = \langle \mu, \diag(X)\rangle.
\] We may now apply Fan's theorem \cite{fan1951maximum}, which implies that the extreme points of 
\[
    \{\diag(X) : X \in \cM_{\sigma}^{(n,m)}\}
\]
are precisely those vectors lying in the set
\[
    \{S \tau \sigma : S \in \Diag(\{-1, 1\}^m), \tau \text{ is a permutation of indices } 1, \ldots , m\},
\]
where $\tau \sigma = (\sigma_{\tau^{-1}(i)})_{i=1}^m$.

If either of the conditions in the statement of the theorem are satisfied, then
\[
    \argmax_y \{\langle \mu, y\rangle : y = \diag(X), X \in \cM_{\sigma}^{(n,m)}\} = \sigma.
\]

Hence, we have that for any maximizer $X$ of $\ell_t$, $\diag(X) = \sigma$, which clearly implies that $X = \Diag(\sigma)$ is the unique maximizer of $\ell_t$.
\end{proof}

\begin{remark}
    \label{rmk:steifel}
    In the special case of the Stiefel manifold, where $\sigma = (1,\dots, 1)$, 
    our assumption  simplifies to requiring that the matrices $A_0, \dots, A_k$ span a subspace of $\R^{n \times m}$ which contains no nonzero matrix that has rank less than $m$.
    Note that the set of matrices in $\R^{n\times m}$ which are singular is of codimension $n-m+1$.
    Therefore, if $k \le n - m$, then for a dense set of tuples $(A_0, \dots, A_k)$, the map $f$ is sectioned. By a limiting argument, this implies that $\Voptf(c)$ is concave for all choices of $A_0, \dots, A_k \in \R^{n\times m}$  provided
    %as long as $\sigma = (1,\dots, 1)$ and 
    $k \le n-m$, recovering a result of \cite{song2024linear}.

    There may exist $k+1$ dimensional subspaces of $\R^{n\times m}$ containing no nonzero matrix of rank less than $m$, even when $k$ is larger than $n-m$; the pairs $(n,k) \in \N^2$ for which there exists such a $k$-dimensional subspace of $\R^{n\times n}$ containing no nonzero singular matrix is closely related to the Radon-Hurwitz numbers~\cite{adams1965matrices}.
\end{remark}

\begin{proof}[Proof of \Cref{thm:singvalues}] 
    Let $A_0, \dots, A_k \in \R^{n \times m}$ span a $\ell$-weakly singularly noncrossing subspace, and let $\sigma \in \R^m$ 
    satisfy $\sigma_{i} = 0$ for $i \ge \ell+1$.

    As in the proof of~\Cref{thm:eigenvalues}, we may assume $A_0,\ldots, A_k$ are linearly independent as elements of the vector space $\R^{n\times m}$.
    \Cref{lem:sing} then implies that the map $f : \cM_{\lambda}^{(n,m)} \rightarrow \R^{k+1}$ defined by $f(X) = (\tr(A_0X), \dots, \tr(A_{k}X))$ is sectioned.
    We may therefore conclude by \Cref{prop:linear_proj} that 
    \[
    V_f(c) = \max \{\langle A_0, X\rangle : \langle A_1, X\rangle = c_1, \dots, \langle A_{k}, X\rangle = c_{k}, X \in \conv(\cM_{\sigma}^{(n,m)})\}.
    \]
\end{proof}

\subsection{A Lie-Theoretic Generalization}\label{subsection:lie}
We discuss how the results pertaining to linear inverse spectral problems earlier in this section can be obtained in a more general setting involving the orbits of Lie groups. 
To avoid introducing too many Lie-theoretic notions, we only give a high level overview of these ideas, and simply indicate how the proofs of \cref{lem:noncrossing} and \cref{lem:sing} generalize to the setting of Lie groups, and reference the works \cite{kostant1973convexity, dadok1985polar, kobert2022spectrahedral} 
for further reading on Kostant's convexity theorem.

First let us describe how the manifolds $\cM_{\lambda}$ and $\cM_{\sigma}^{(n,m)}$ can be thought of as the orbits of group actions.
In the case of the 
linear inverse eigenvalue problem, the manifold $\cM_{\lambda}$ can be thought of as the orbit of the matrix $\Diag(\lambda)$ under the action of the group $\SO(n)$, acting by letting $U \cdot X = UXU^{\intercal}$.
For the linear inverse singular value problem with $m < n$, the manifold $\cM^{(n,m)}_{\sigma}$ is the orbit of the matrix $\Diag(\sigma)$ under the action of the group $\SO(n) \times \SO(m)$ acting by letting $(U, V) \cdot X = UXV^{\intercal}$.

Thus, we see that in both of these cases, there is an underlying group $G$ acting orthogonally on some representation $V$ for which the domain $\cM$ is the orbit of an element under the group action. Indeed, in both of these examples, the underlying group representation is a `polar representation', which were defined in \cite{dadok1985polar}. While we will not define a polar representation precisely, we will note that a polar representation involves two pieces of information in addition to the group action: a linear subspace of $V$ called the Cartan subspace and a group $W$ acting on the Cartan subspace called the Weyl group. In the case of $\SO(n)$ acting on $\R^{n\times n}_{sym}$, the Cartan subspace is the space of diagonal matrices, and the group $W$ is $\mathfrak{S}_n$, the permutation group on $n$ letters. In the case of $\SO(n) \times \SO(m)$ acting on $\R^{n\times m}$, the Cartan subspace is the set of diagonal matrices in $\R^{n \times m}$, and the Weyl group is a semidirect product $\Z_2^{m} \rtimes \mathfrak{S}_n$ known as the hyperoctahedral group.

For polar representations, there is a common generalization of the Schur-Horn theorem and Fan's theorem known as Kostant's convexity theorem.
\begin{theorem}[Kostant Convexity \cite{kostant1973convexity}]
If $(V,\cdot)$ is a polar representation of a compact connected Lie group $G$, $\Sigma$ is the Cartan subspace, and $W$ is the Weyl group, then for any $v\in \Sigma$,
\[
    \pi_{\Sigma}(G \cdot v) = \conv(W\cdot v),
\]
where $\pi_{\Sigma}$ denotes the orthogonal projection onto the subspace $\Sigma$.
\end{theorem}
If we consider the proofs of \cref{lem:noncrossing} and \cref{lem:sing}, we see that we can replace the usage of the Schur-Horn theorem and Fan's theorem with this result immediately.
It may be noted that our proofs only use use the `easy' part of the Kostant convexity theorem that $\pi_{\Sigma}(G \cdot v) \subseteq \conv(W\cdot v)$, which can be proven using elementary calculus.

In order to fully generalize  \cref{lem:noncrossing,lem:sing} to the context of polar representations, we would need to define the relevant version of a noncrossing subspace.
To derive this generalization, we will need to recall that if $v \in \Sigma$, then the stabilizer of $v$ inside $W$ is the subgroup $Stab(v) \subseteq W$ of group elements fixing $v$.
We will also require the notion of a closed Weyl chamber. One way to define a closed Weyl chamber in the context of convex geometry is that it is the normal cone to $\conv(W \cdot v)$ at $v$ for any $v \in \Sigma$ which has trivial stabilizer in $W$. For a fixed closed Weyl chamber $C \subseteq \Sigma$, for any $A \in v$ there is a unique $\sigma(A) \in C$ so that there exists some $g \in G$ with $A = g \cdot \sigma(A)$. In our running example of the action of $\SO(n)$ on $\R^{n\times n}_{sym}$, a Weyl chamber inside the set of diagonal matrices is simply the set of diagonal matrices with nonincreasing entries along the diagonal.

With these notions in mind, fix a polar representation $V$ with a Cartan subspace $\Sigma$, a Weyl chamber $C \subseteq \Sigma$, and Weyl group $W$. If $H$ is a subgroup of $W$, then we will say then that a linear subspace $S \subseteq V$ is \emph{$H$-noncrossing} if  $Stab(\sigma(A)) \subseteq H$ for every nonzero $A \in S$. In our running example of the action of $\SO(n)$ on $\R^{n\times n}_{sym}$, if $H = \mathfrak{S}_{n-\ell}$ is the subgroup of $\mathfrak{S}_n$ that fixes the first $\ell$ elements of $[n]$, then $S \subseteq \R^{n\times n}_{sym}$ is $H$-noncrossing if and only if $S$ is $\ell$-weakly noncrossing. Note that in this example, $H = Stab((v_1, \dots, v_{\ell}, 0, \dots, 0))$ whenever $v_1, \dots, v_{\ell}$ are distinct.

\Cref{lem:noncrossing,lem:sing} on sectioning can be generalized to the following:

\begin{lemma}
\label{lem:noncrossing_gen}
Fix a polar representation $V$ with a Cartan subspace $\Sigma$, a Weyl chamber $C \subseteq \Sigma$, and Weyl group $W$. Also fix some $v \in C$ and let $\cM = G \cdot v$.
If $A_0, \dots, A_k \in V$ are linearly independent elements of a $Stab(v)$-noncrossing subspace, then the linear map $f(x) = (\langle A_0, x\rangle, \dots, \langle A_k, x\rangle)$ is sectioned on $\cM$.
\end{lemma}

The proof is essentially identical to those given for \Cref{lem:noncrossing,lem:sing}, but we replace the Schur-Horn and Fan's theorem with Kostant's convexity theorem.

Finally, we note that \cref{thm:eigenvalues,thm:singvalues} also can be generalized by combining \cref{thm:hidden_cvx} with \cref{lem:noncrossing_gen}. Indeed, it is even known that in the context of polar representations that $\conv(G\cdot v)$ is in fact spectrahedral, and therefore, this convex reformulation can be expressed as a semidefinite program.

\begin{theorem}
Fix a polar representation $V$ with a Cartan subspace $\Sigma$, a Weyl chamber $C \subseteq \Sigma$, and Weyl group $W$. Also fix some $v \in C$ and let $\cM = G \cdot v$.
If $A_0, \dots, A_k \in V$ are linearly independent elements of a $Stab(v)$-noncrossing subspace of $V$, then for any $c \in \R^k$
\begin{align*}
    &\max \{\langle A_0, X\rangle : \langle A_1, X\rangle = c_1, \dots, \langle A_k, X\rangle = c_k, X \in \cM\}\\
    &=\max \{\langle A_0, X\rangle : \langle A_1, X\rangle = c_1, \dots, \langle A_k, X\rangle = c_k, X \in \conv(\cM)\}.
\end{align*}
\end{theorem}

\section{Unbalanced Procrustes Problems}\label{sec:UPP}
This section will analyze the UPP from the perspective of Lagrangian dual bundles.
To restate the UPP, we are given two matrices $U \in \R^{n \times d}$ and $W \in \R^{m \times d}$ and we want to find the $X \in \St^{n, m}$ minimizing the objective $\|U^{\intercal}X-W^{\intercal}\|^2.$

We often make use of an alternative formulation of this problem in terms of two matrices $A \in \R^{n \times n}_{sym}$ and $B \in \R^{n \times m}$.
\begin{lemma}
\label{lem:UPP_equivalent}
Given $U\in \mathbb{R}^{n\times d}$ and $W\in \mathbb{R}^{m\times d}$, 
let $A = -UU^{\intercal}$ and $B = 2UW^{\intercal}$.
Then a matrix $X \in \St^{n,m}$ minimizes $\|U^{\intercal}X-W^{\intercal}\|^2$ over $\St^{n,m}$ if and only if $X$ maximizes $\langle A, XX^{\intercal}\rangle + \langle B, X\rangle$ over $\St^{n,m}$.
\end{lemma}
\begin{proof}
    Note that 
    \[
        \|U^{\intercal}X-W^{\intercal}\|^2 = \|U^{\intercal}X\|^2 - 2 \langle UW^{\intercal}, X\rangle + \|W\|^2  = - \langle A,XX^{\intercal}\rangle - \langle B, X\rangle + \|W\|^2,
    \]
    and that the term $\|W\|^2$ is a constant independent of $X$.
\end{proof}

\begin{example}[UPP and the generalized trust region subproblem]
\label{ex:UPP1}
When $m=1$, the UPP is equivalent to the generalized trust region subproblem of minimizing an inhomogeneous quadratic function on the sphere. We will show how this case can be analyzed with our theory.

In the case $m= 1$, $St^{n,1} = S^{n-1}$ and we wish to maximize
\[
    \max_{x \in S^{n-1}}\;\langle B, x\rangle + x^{\intercal}Ax,
\]
where $B \in \R^n$ and $A \in \R^{n\times n}_{sym}$ are given.
We want to show that the map 
\[
f : S^{n-1} \rightarrow \R^2,\quad x\mapsto  (\langle B, x\rangle, x^{\intercal}Ax)
\] 
is sectioned as long as $\{A, BB^{\intercal}\}$ spans a 1-weakly noncrossing subspace. (Note that this condition is very weak, as there is a dense set of pairs $(A, B) \in \R^{3\times 3} \times \R^{3\times 2}$ satisfying this condition.)
To do this, we first analyze the modified function $\hat{f}(x) =(\langle B, x\rangle^2, x^{\intercal}Ax)$. This modified function is convenient, because  $\langle B, x\rangle^2 = x^{\intercal}BB^{\intercal}x$, and so the function $\langle t, \hat{f}(x)\rangle = x^{\intercal}(t_0A + t_1BB^{\intercal})x$ 
is a quadratic form for any $t \in \R^{2}_{0+}$.

The largest eigenvalue of the matrix $t_0 A + t_1 BB^{\intercal}$ is non-degenerate for all $t_0 \ge 0$ and  $t_1 \in \R$ when $\{A, BB^{\intercal}\}$ spans a 1-weakly noncrossing subspace. It follows that the function $\hat{f}$ is uniquely maximized.
Since $\sqrt{\cdot}$ is strictly concave, we may apply \Cref{lem:concave_composition} to see that $f$ is also uniquely maximized, and thus sectioned.
\end{example}

\subsection{Lagrangian Dual Sections for UPP when $(n,m) = (3,2)$}

We aim to provide a
condition on $A \in \R_{sym}^{3 \times 3}$ and $B \in \R^{3 \times 2}$ which ensures that for any $t \in \R$, the function 
\[
\langle t, f(X)\rangle = t_0\langle A, XX^{\intercal}\rangle + t_1\langle B, X\rangle
\]
has a unique maximizer over $X \in \St^{3,2}$. \Cref{lem:unique_max} then implies that $f(X) = (\langle B, X\rangle, \langle A, XX^{\intercal}\rangle)$ is sectioned in this case.

We will require a technical definition in order to state our claim.

\begin{definition}\label{def:well-positioned}
Fix a pair $(A_0,A_1) \in \R^{3\times 3}_{sym} \times \R^{3\times 3}_{sym}$.
Let $w \in S^{2}$ be an
eigenvector of $A_0$ associated to the largest eigenvalue.

If $(A_0, A_1) \in \R^{3\times 3}_{sym} \times \R^{3\times 3}_{sym}$ %satisfy 3 
satisfies three
conditions, we will say that this pair is \emph{well-positioned}.
\begin{enumerate}
    \item $A_0,A_1$ span a 2-dimensional 1-weakly noncrossing subspace.
    \item $A_1$ has no eigenvector orthogonal to $w$ (and in particular, does not have $w$ as an eigenvector).
    \item
    The matrix $A_0 + t^* A_1 - \lambda^* I$ has a positive eigenvalue, where
    $z^* \in \R^3$ is any nonzero solution to the linear equations $w^{\intercal} A_1z^* = w^{\intercal}z^* = 0$,
    and where  $t^*, \lambda^*$ are scalars uniquely determined by $z^*$ and the following equations:
    \[
        (A_0 + t^* A_1 - \lambda^* I)z^* = 0.  
    \]
\end{enumerate}
\end{definition}
\begin{remark}[Remark on \cref{def:well-positioned}]
    Conditions 1 and 2 above are generic in the sense that they hold for all $(A_0, A_1)$ in a dense subset of $\R^{3\times 3}_{sym} \times \R^{3\times 3}_{sym}$.
    Condition 3, however, is not generic in this sense.

    Assuming that  Condition 2 holds, we can see that there is a one dimensional space of solutions to the linear equations $w^{\intercal} A_1z = w^{\intercal}z^* = 0$ by noting that $A_1 w$ and $w$ are only linearly dependent if $w$ is an eigenvector of $A_1$, which we have forbidden. For a fixed $z^*$, we note that the equation $(A_0 + t^* A_1 - \lambda^* I)z^* = 0$ can be viewed as a system of 3 equations in 2 variables, but where the one of the equations is redundant since $w^{\intercal}(A_0 + t A_1 - \lambda I)z^* = 0$ for any $t, \lambda \in \R$. To see that there is exactly one solution to the linear system, note that if there were infinitely many or zero, then there would need to be a solution to the homogenized system $(t A_1 - \lambda I)z^* = 0$, which implies that $z^*$ is an eigenvector of $A_1$. Since $z^*$ is orthogonal to $w$ by definition, this would be a violation of the second assumption.
\end{remark}

\begin{theorem}
\label{thm:UPP32}
    Let $B \in \R^{3 \times 2}$, and 
    $B = \begin{pmatrix} \sigma_1 & 0\\ 0&\sigma_2\\0&0\end{pmatrix}$, where $\sigma_1 > \sigma_2 \ge 0$.
    
    Suppose the pair $(-BB^{\intercal}, A)$ is well-positioned, and the linear systems
    \[
        (t A_1 + \lambda I - BB^{\intercal})
        \begin{pmatrix}
        \sqrt{1-\frac{\sigma_2}{\sigma_1}} \\ 0 \\ \frac{\sigma_2}{\sigma_1}
        \end{pmatrix} = 2\sigma_1\sigma_2
        \begin{pmatrix}0\\0\\1\end{pmatrix}
    \]
    \[
        (t A_1 + \lambda I - BB^{\intercal})
        \begin{pmatrix}
        -\sqrt{1-\frac{\sigma_2}{\sigma_1}} \\ 0 \\ \frac{\sigma_2}{\sigma_1}
        \end{pmatrix} = 2\sigma_1\sigma_2
        \begin{pmatrix}0\\0\\1\end{pmatrix}
    \]
    have no solutions in $t$ and $\lambda$.
    Then the function 
    \[
    f(X) = (\langle B, X\rangle, \langle A, XX^{\intercal}\rangle)
    \]is uniquely maximized on $\St^{3,2}$. In particular, if $(A,BB^{\intercal})$ is well-positioned, then $f$ is sectioned.
\end{theorem}

Note that, for arbitrary $B \in \R^{3\times 2}$, we can put $B$ in the form required by~\Cref{thm:UPP32} by expressing $A$ and $B$ in a basis of singular vectors for $B$.

At a high level, the proof of~\Cref{thm:UPP32} reduces to proving an analogous property for an inhomogeneous quadratic maximization problem on the sphere. More precisely, we first reduce UPP for well-positioned inputs to an optimization problem over the Grassmannian.
Subsequently, we reduce the $(n,m) = (3,2)$ Grassmannian problem to optimization over the real projective plane $\PP^2.$ We will define these topological spaces when they become relevant to our analysis.

We will establish our results in the order in which they are necessary to prove \Cref{thm:UPP32}, before presenting a proof in \Cref{subsec:UPP32}.

\subsection{A Quadratic Maximization Problem}

We recall the description of real projective space $\PP^{n-1} = S^{n-1} / \{z\sim -z\}$ obtained by identifying antipodal points $z$ and $-z$ on the sphere.

\begin{lemma}
    \label{lem:inhom_quad}
    Let $Q \in \R^{n \times n}_{sym}$ and $w \in \R^n$. If $Q$ has no repeated eigenvalues and $w$ is not orthogonal to the eigenvector of $Q$ with largest eigenvalue,
    then the function $\ell(z) = z^{\intercal}Qz + |w^{\intercal}z|$ has a unique maximum over $z \in \PP^{n-1}$.

\end{lemma}
\begin{proof}

We aim to analyze the maximizer(s) of the function 
    \[
        \ell(z) = z^{\intercal}Qz + \sqrt{z^{\intercal}ww^{\intercal}z},
    \]
where $z\in \PP^{n-1}$. 
Define the set
    \[
        \cC = \{(z^{\intercal}Qz, z^{\intercal}ww^{\intercal}z) \in \mathbb{R}^2 : z \in \PP^{n-1}\}.
    \]
By Brickman's Theorem~\cite[Theorem 14.1]{barvinok}, $\cC$ is convex.
We let $\hat\ell(c)=c_0+\sqrt{c_1}$. 
Then $\ell(z)=\hat \ell(c)$ for $c=(z^{\intercal}Qz , \sqrt{z^{\intercal}ww^{\intercal}z})\in \cC$.

Note that $\hat \ell(c)$ is strictly concave in $c_1$. Therefore, 
the maximizers $(c_0^*,c_1^*)$ of $\hat \ell(c)$ over $\cC$ must all have a unique value of $c_1^*$.
It follows that any $z^*$ maximizing $\ell(z)$ 
satisfies
$z^{*\intercal}ww^{\intercal}z^* = (w^{\intercal}z^*)^2 = c_1^*$.

    We may consider two cases: 
if $c_1^* = 0$, then the second term in $\ell(z)$ vanishes as 
 $w^{\intercal}z^* = 0$.
This implies $z^*$ is a maximizer of the quadratic form $z^\intercal Qz$, and therefore an eigenvector corresponding to the largest eigenvalue of $Q$. 
If there are multiple maximizers of $\hat \ell(c)$, then $Q$ must have degenerate maximum eigenvalue, a contradiction.

If $c_1^* \neq 0$, then $\hat{\ell}(c)$ is differentiable at any maximizer $c^*$, and we may compute $\nabla \hat{\ell}(c) = \left(1, \frac{1}{2\sqrt{c_1^*}}\right)$.
The first order optimality condition on $\hat{\ell}(c)$ implies that any maximizer $c^*$ of $\hat{\ell}(c)$ must in fact maximize the linear function $ c_0 + \frac{c_1}{2\sqrt{c_1^*}}$ on $\cC$. This implies that $z^*$ must be an eigenvector of the matrix 
\[
    M = Q + \frac{ww^{\intercal}}{2\sqrt{c_1^*}}
\]
associated to the eigenvalue $\lambda$, the largest eigenvalue of $M$.
If we choose the representative of $z^*$ so that $w^{\intercal}z^* = \sqrt{c_1^*}$, 
then the eigenvector equation becomes
\[
    Mz^* = Qz^* + \frac{w}{2} = \lambda z^*,
\]
and we see that for any maximizer $z^*$ of $\ell(z)$, we have that 
\[
    (Q-\lambda I)z^* = - \frac{w}{2}.
\]
In order for this to have multiple solutions, we see that $Q-\lambda I$ must be singular, i.e. that $\lambda$ must also be an eigenvalue of $Q$, and we must also have that $w$ is in the image of $(Q-\lambda I)$, i.e. $w$ must be orthogonal to all eigenvectors of $Q$ with eigenvalue $\lambda$. Note that $\lambda$ must in fact be the largest eigenvalue of $Q$, since it is the largest eigenvalue of $M$ and $M \succeq Q$. This is a contradiction, as we have assumed that $w$ is not orthogonal to the eigenvectors of $Q$ with largest eigenvalue.
\end{proof}

We will need to understand when inhomogeneous quadratic optimization problems have a unique maximum \emph{for all $Q$ in a linear matrix pencil}. This has a connection to our notion of a well positioned pair of matrices, as we show in the next lemma.
\begin{lemma}
    \label{lem:inhom_quad_pencil}
    Suppose that $(A_0,A_1) \in \R^{3\times 3}_{sym}\times\R^{3\times 3}_{sym}$ is well-positioned, and that $w$ is the eigenvector of $A_0$ associated to the largest eigenvector of $A_0$.
    For $t \in \R^2$, let $Q(t) = t_0A_0 + t_1 A_1 \in \R^{n\times n}_{sym}$.
    For all $t\in \R^2_{0+}$,
    the function $\ell_t(z) = z^{\intercal}Q(t)z + |w^{\intercal}z|$
    has a unique maximizer over $z\in \mathbb{P}^2$.
\end{lemma}
\begin{proof}

    We argue that if $(A_0, A_1)$ are well-positioned, then $Q(t)$ has no repeated eigenvalues for all $t \in \R^2_{0+}$, and $w$ is not orthogonal to the eigenvector of $Q(t)$ with maximum eigenvalue. This implies by \Cref{lem:inhom_quad} that for all $t \in \R^2_{0+}$, the function $\ell_t$ has a unique maximizer on the sphere. Note that the matrices $Q(t)$ have no repeated eigenvalues, since $A_0$ and $A_1$ span a noncrossing subspace.

    We will first consider the set of $t \in \R^2$ so that $w$ is orthogonal to some eigenvector of $Q(t)$. This is equivalent to the condition that there exists $z \in \PP^2$ and $\lambda \in \R$ so that
    \begin{align}
    \label{eq:eigenvec_ortho}
        Q(t) z  = \lambda z\\
        w^{\intercal} z = 0.\notag
    \end{align}
    Note that this in particular implies that $z$ solves the equation
    \[
        w^{\intercal} Q(t) z  = t_0w^{\intercal}A_0z + t_1w^{\intercal}A_1 z = t_1w^{\intercal}A_1 z = 0.
    \]
    For the second equality, we used the fact that $w$ is an eigenvector of $A_0$ to conclude that $w^{\intercal}A_0z = w^{\intercal}z = 0$.
    Hence, we see that  either $t_1 = 0$, or $w^{\intercal}A_1 z = 0$. In case $t_1 = 0$, we see that $Q(t) = t_0 A_0$, where $t_0 > 0$. Since $w$ is the eigenvector associated to the largest eigenvalue of $A_0$, we see in particular that $w$ is not orthogonal to the largest eigenvector of $t_0A_0$.

    If $w^{\intercal}A_1z = w^{\intercal}z = 0$, then this implies that $z$ is a multiple of $z^*$ appearing in the definition of `well-positioned' (\Cref{def:well-positioned}.) Equations~\eqref{eq:eigenvec_ortho} thus are equivalent to
    \begin{align*}
        (t_0 A_0 + t_1 A_1 - \lambda I) z^* = 0.
    \end{align*}
    Up to scaling, we may assume $t_0 = 1$, which then implies $t_1 = t^*$ and $\lambda = \lambda^*$ in~\Cref{def:well-positioned}.
    We then see from this definition that $Q(t) - \lambda I$ has a positive eigenvalue, implying that $\lambda$ is not the largest eigenvalue of $Q(t)$.
    Thus, for all $t \in \R^2_{0+}$, the conditions of \cref{lem:inhom_quad} hold.
    
    We conclude that for all $t \in \R^2_{0+}$, $\ell_t$ has a unique maximizer, as desired.
\end{proof}

\subsection{A Grassmannian Optimization Problem}
We consider the Grassmannian to be the set of symmetric orthogonal projectors of rank $m$ in $\R^{n\times n}_{sym}$:
\[
    \Gr^{n,m} = \{Z \in \R^{n \times n}_{sym} : Z^2 = Z,\, \tr(Z) = m\}.
\]
Fix $A \in \R^{n\times n}_{sym}$ and $B \in \R^{n \times m}$.
Consider the function
\begin{equation}\label{eq:f-tick-goal}
\hat{f}:\Gr^{n,m}\to\mathbb{R},\quad
    Z \mapsto (\|B^{\intercal}Z\|_1, \langle A, Z\rangle),
\end{equation}
where $\|\cdot\|_1$ 
denotes the trace norm (the sum of a matrix's singular values).

Our aim is to show 
the function in \eqref{eq:f-tick-goal}
is uniquely maximized when $(n,m) = (3,2)$ and $(-BB^{\intercal}, A)$ satisfies the conditions of \cref{thm:UPP32}.
To do this, we will apply \cref{lem:concave_composition} to the following function, which we show is uniquely maximized on $\Gr^{3,2}$:
\begin{equation}
\label{eq:def_tilde_f}
    \tilde{f}(Z) = (\|B^{\intercal}Z\|_1^2, \langle A, Z\rangle).
\end{equation}

Before we prove these results, we will need some algebraic lemmas. 
\begin{lemma}\label{lemma:norm-BZ-n-l}
    Suppose that $B \in \R^{n \times 2}$, and that $\Pi_B \in \R^{n \times 2}$ is the matrix of left singular vectors of $B$.
    If $Z \in \Gr^{n,2}$, then 
    \begin{equation}\label{eq:norm-BZ-n-2}
        \|B^{\intercal}Z\|_1 = \sqrt{\tr(BB^{\intercal}Z) + 2\sigma_1(B) \sigma_2(B)\sqrt{\det(\Pi_B^{\intercal}Z\Pi_B)}}.
     \end{equation}
\end{lemma}
\begin{proof}
    Throughout the proof, for a $n\times 2$ matrix $M$, we let $\sigma_1(M)$ and $\sigma_2(M)$ denote the first and second singular values of $M$.
    By definition, 
    \[
        \|B^{\intercal}Z\|_1 = \sigma_1(B^{\intercal}Z) + \sigma_2(B^{\intercal}Z).
    \]
    Squaring each side, we get
    \[
        \|B^{\intercal}Z\|_1^2 = \sigma_1(B^{\intercal}Z)^2 + \sigma_2(B^{\intercal}Z)^2 + 2\sigma_1(B^{\intercal}Z)\sigma_2(B^{\intercal}Z).
    \]
    We note that 
    \[
        \sigma_1(B^{\intercal}Z)^2 + \sigma_2(B^{\intercal}Z)^2 = \|B^{\intercal}Z\|_2^2 = \tr(B^{\intercal}ZB),
    \]
    where we 
    used the fact that $Z^2 = Z$.
    % have noted that $Z^2 = Z$ since $Z \in \Gr^{3,2}$.

    We also note that
    \[
        \sigma_1(B^{\intercal}Z)\sigma_2(B^{\intercal}Z) = \sqrt{\det(B^{\intercal}ZB)}.
    \]

    Finally, we apply singular value decomposition of $B$ in order to write
    \[
        B = \Pi_B \begin{pmatrix}\sigma_1(B) & \\ & \sigma_2(B)\end{pmatrix} V^{\intercal},
    \]
    where $V \in O(2)$ and note using the multiplicative property of determinants that
    \[
        \det(B^{\intercal}ZB) =  \sigma_1(B)^2\sigma_2(B)^2\det(V)^2\det(\Pi_B^{\intercal}Z\Pi_B).
    \]
    The result follows.
\end{proof}

\begin{lemma}\label{lemma:norm-BZ-n-2}
    Suppose that $\Pi \in \R^{n \times (n-1)}$ is a matrix such that
    $\det(\Pi^{\intercal}\Pi) = 1$.
    If $z\in S^{n-1}$,  then
    \[
    \det(\Pi^{\intercal}(I-zz^{\intercal})\Pi) = z^{\intercal}(I-\Pi\Pi^{\intercal})z.
     \]
\end{lemma}

\begin{proof}
    This follows from the Matrix Determinant lemma, which implies that
    \[
    \det(\Pi^{\intercal}(I-zz^{\intercal})\Pi) = \det(\Pi^{\intercal}\Pi)(1- z^{\intercal}\Pi\Pi^{\intercal}z) = z^{\intercal}(I-\Pi\Pi^{\intercal})z.
     \]
\end{proof}

\begin{lemma}
\label{lem:f_tilde_um}
    Suppose that $A \in \R^{3 \times 3}_{sym}$ and $B \in \R^{3 \times 2}$ satisfy the property that $(-BB^{\intercal}, A)$ is well-positioned. Then the function $\tilde{f}$ defined in \eqref{eq:def_tilde_f} is uniquely maximized.
\end{lemma}
\begin{proof}
    Because $k = 1$ in this setting, it suffices to show that for any $t \in \R^{2}_{0+}$, we have that the function
    \[
        \ell_t(Z) = \langle t, \tilde f(Z) \rangle = t_0\|B^{\intercal}Z\|_1^2 + t_1\langle A, Z\rangle
    \]
    has a unique maximizer in $\Gr^{3,2}$.
    
    Note that any matrix $Z \in \Gr^{3,2}$ can be uniquely represented as $I - zz^{\intercal}$ for $z \in \PP^2$.
    Combining \Cref{lemma:norm-BZ-n-2,lemma:norm-BZ-n-l}, we obtain that 
    \begin{align*}
        \ell_t(I-zz^{\intercal}) 
        &= t_0\|B^{\intercal} (I-zz^{\intercal})\|_1^2 + t_1\langle A, I-zz^{\intercal}\rangle\\
        &= \langle t_0BB^{\intercal}+t_1A, (I-zz^{\intercal})\rangle + 2\sigma_1(B)\sigma_2(B)t_0\sqrt{z^{\intercal}(I-\Pi_B\Pi_B^{\intercal})z}
    \end{align*}
    Note that because $\Pi_B\Pi_B^{\intercal}$ is a rank 2 orthogonal projector, we have $I - \Pi_B\Pi_B^{\intercal} = ww^{\intercal}$ for some $w \in \R^3$, where $w$ is orthogonal to the image of $B$.
    
    If we let $Q(t) = -t_0BB^{\intercal} - t_1A$, then the maximizers of $\ell_t(I-zz^{\intercal})$ are in 1-1 correspondence with the maximizers of the function
    \[
        z^{\intercal}Q(t)z + 2\sigma_1(B)\sigma_2(B)t_0|w^{\intercal}z|.
    \]
    \Cref{lem:inhom_quad_pencil} implies that the $\ell_t$ has a unique maximizer for all $t \in \R^2_{0+}$. This shows that $f$ is uniquely maximized, as desired.
\end{proof}
\begin{lemma}
    If $A \in \R^{3\times 3}_{sym}$ and $B \in \R^{3\times 2}$ satisfy the property that $(-BB^{\intercal}, A)$ is well positioned, then the function $\hat{f}(Z) = (\|B^{\intercal}Z\|_1, \langle A, Z\rangle)$ is uniquely maximized over the Grassmannian.
\end{lemma}
\begin{proof}
    Note that
    \[
        (\|B^{\intercal}Z\|_1, \langle A, Z\rangle) = 
        (\sqrt{\|B^{\intercal}Z\|_1^2}, \langle A, Z\rangle),
    \]
    and so since $\sqrt{\cdot}$ is a strictly concave continuous function, \cref{lem:concave_composition} yields that $\hat{f}(Z)$ is uniquely maximized as long as $\tilde{f}(Z)$ is. The fact that $\tilde{f}(Z)$ is uniquely maximized is precisely \cref{lem:f_tilde_um}.
\end{proof}
\subsection{Proof of Unique Maximization when $(n,m) = (3,2)$}
\label{subsec:UPP32}
When $A$ and $B$ satisfy the conditions of \cref{thm:UPP32}, we will show that the function $f(X) = (\langle B, X\rangle, \langle A, XX^{\intercal})$ is uniquely maximized over $\St^{3,2}.$

We will first note that there is an action of the group $O(2)$ on $\St^{3,2}$ by taking $(X,U) \in \St^{3,2} \times O(2)$ to $X U$.
We can characterize the orbits of this action using the Grassmannian $\Gr^{3,2}$: the relation $X = X'U$ for some $U \in O(2)$ is equivalent to the claim that $XX^{\intercal} = X'X'^{\intercal}$.
This action preserves the expression $\langle A, XX^{\intercal}\rangle$, so maximizing $f(X U)$ and $\langle B, XU\rangle $ over $U \in O(2)$ are equivalent.
\begin{lemma}
    \label{lem:stiefel_to_grassmannian}
    For any $B \in \R^{3 \times 2}$ and $X \in \St^{3,2}$, let  $Z = XX^{\intercal}$. We have
\begin{equation}\label{eq:stiefel_to_grassmannian}
    \max_{U \in O(2)} \langle B, XU\rangle = \|B^{\intercal}Z\|_1.
    \end{equation}    
    If further, $B = \begin{pmatrix} \sigma_1 & 0 \\ 0 & \sigma_2 \\ 0 & 0\end{pmatrix}$ (i.e. we are the singular vector basis for $B$), then there is a unique $U$ achieving this maximum value as long as $Z$ is not of the form
    \[
        I -
        \begin{pmatrix}
        \pm \frac{\sqrt{\sigma_1^2 - \sigma_2^2}}{\sigma_1}\\ 0\\\pm \frac{\sigma_2}{\sigma_1}
        \end{pmatrix}
        \begin{pmatrix}
        \pm \frac{\sqrt{\sigma_1^2 - \sigma_2^2}}{\sigma_1}\\ 0\\\pm \frac{\sigma_2}{\sigma_1}
        \end{pmatrix}^{\intercal}.
    \]
\end{lemma}
\begin{proof}
    We can once again apply Fan's theorem to state 
    \[
        \max_{U \in O(2)}\langle B, XU\rangle = 
        \max_{U \in O(2)}\langle X^{\intercal}B, U\rangle =
        \sigma_1(X^{\intercal}B) + \sigma_2(X^{\intercal}B).
    \]
    Note that because 
    \[
    (X^{\intercal}B)^{\intercal}X^{\intercal}B = 
    B^{\intercal}ZB = 
    (ZB^{\intercal})^{\intercal}ZB,
    \]
    we have that the singular values of $X^{\intercal}B$ are the same as those of $B^{\intercal}Z$, giving~\eqref{eq:stiefel_to_grassmannian}.

    Fan's theorem also implies that the optimal $U$ in~\eqref{eq:stiefel_to_grassmannian} is unique as long as the singular values of $X^{\intercal}B$ are distinct, which is true if and only if the eigenvalues of $B^{\intercal}ZB$ are distinct. Since $B^{\intercal}ZB$ is a $2\times 2$ matrix, the eigenvalues of this matrix are not distinct if and only if $B^{\intercal}ZB = cI$  for some $c \in \R$.
    Assume this holds.

    We can write $Z = I - zz^{\intercal}$ for a unique $z \in \PP^2$, from which it follows that
    \[
        \begin{pmatrix}\sigma_1 z_1\\ \sigma_2z_2\end{pmatrix}
        \begin{pmatrix}\sigma_1 z_1\\ \sigma_2z_2\end{pmatrix}^{\intercal} =
        \begin{pmatrix}
            \sigma_1^2 - c & 0 \\ 0 & \sigma_2^2 - c
        \end{pmatrix}.
    \]
    Since the left hand side is a rank one PSD matrix, the right hand side would need to be as well, i.e. $c = \sigma_2^2$, and we see that
    $z_1 = \pm \frac{\sqrt{\sigma_1^2 - \sigma_2^2}}{\sigma_1}$, $z_2 = 0$, and $z_3 = \pm \frac{\sigma_2}{\sigma_1}$.
    The conclusion of the theorem follows.
\end{proof}

\begin{proof}[Proof of \cref{thm:UPP32}]
    We have seen that if $(-BB^{\intercal},A)$ is well-positioned, then the function $f'(Z) = (\|B^{\intercal} Z\|_1, \langle A, Z\rangle)$ is uniquely maximized, i.e. for any $t \in \R$, the function
    \[
        \ell'_t(Z) = \|B^{\intercal} Z\|_1+t\langle A, Z\rangle
    \]
    has a unique maximizer $Z(t)$. 

    For any fixed $t$, if $X$ maximizes $\langle t, f(X)\rangle = t_0\langle B, X\rangle + t_1\langle A, XX^{\intercal}\rangle$, then \Cref{lem:stiefel_to_grassmannian} implies that $Z = XX^{\intercal}$ maximizes the linear function $t_0\|B^{\intercal}Z\|_1 + t_1\langle A, Z\rangle$. The uniqueness of the maximizer of this function implies that any maximizer of $\langle t, f(X)\rangle$ must satisfy $XX^{\intercal} = Z(t)$ and that 
    $X \in \argmax \{\langle B, X\rangle : XX^{\intercal} = Z(t)\}$. Thus, if we can show that $Z(t)$ is not of the form 
    \[
        I -
        \begin{pmatrix}
        \pm \frac{\sqrt{\sigma_1^2 - \sigma_2^2}}{\sigma_1}\\ 0\\\pm \frac{\sigma_2}{\sigma_1}
        \end{pmatrix}
        \begin{pmatrix}
        \pm \frac{\sqrt{\sigma_1^2 - \sigma_2^2}}{\sigma_1}\\ 0\\\pm \frac{\sigma_2}{\sigma_1}
        \end{pmatrix}^{\intercal}
    \]
    for any $t \in \R^2_{0+}$, then the theorem follows by \cref{lem:stiefel_to_grassmannian}.

    We note that we may write $Z(t) = I-z(t)z(t)^{\intercal}$ for each $t$, and so we equivalently want to show that $z(t) \neq \begin{pmatrix} \pm \frac{\sqrt{\sigma_1^2 - \sigma_2^2}}{\sigma_1}\\ 0\\\pm \frac{\sigma_2}{\sigma_1} \end{pmatrix}$ for any $t \in \R^2_{0+}$.
    
    We have seen that each $z(t)$ is a maximizer of the function $z^{\intercal}(t_0'BB^{\intercal} + t_1' A)z + 2\sigma_1\sigma_2|w^{\intercal}z|$ on $\PP^{n-1}$, for some $t' \in \R^2_{0+}$, and where $w$ is orthogonal to the image of $B$, and $\sigma_1, \sigma_2$ denote the singular values of $B$.

    To conclude, in order for this value of $z(t)$ to be a critical point for the function $z^{\intercal}(-BB^{\intercal} + t'_1A)z + 2\sigma_1\sigma_2|w^{\intercal}z|$, we have seen that there would need to be a solution to one of the two linear systems of equations
    \[
        (BB^{\intercal} + t' A + \lambda I)\begin{pmatrix} \pm \frac{\sqrt{\sigma_1^2 - \sigma_2^2}}{\sigma_1}\\ 0\\ \frac{\sigma_2}{\sigma_1} \end{pmatrix} + 2\sigma_1\sigma_2w.
    \]
    Thus, our assumption that there is no solution to these equations translates directly to the fact that for every $t$, there is a unique $X$ satisfying the conclusion of the theorem.
\end{proof}

\section{Algorithms for Riemannian Manifolds}\label{sec:alg-in-Riemannian-case}
Given a topological space $\cM$ and a function $f : \cM\rightarrow \R^{k+1}$ that
has a Lagrangian dual section $D(t)$, two algorithmic questions immediately arise from our study. First, we may ask whether it is possible to maximize the Lagrangian  $\langle t, f(x)\rangle$ for arbitrary $t \in \tDomain$. Second, we may ask whether it is possible to solve the constrained optimization problem \eqref{eq:constrained}, given an oracle for maximizing Lagrangians for arbitrary choices of $t$. 

For the first question, let us assume that we have access to the value of $D(t^{(0)})$ for some fixed $t^{(0)} \in \tDomain$, 
and that we want to find $D(t^{(1)})$ for some $t^{(1)} \neq t^{(0)}$. Given that $D(t)$ is a continuous function, it is natural to attempt a `path tracking' approach to solving this problem, in which we slowly adjust the parameter $t$ from $t_0$ to $t_1$, while repeatedly finding $D(t)$ using some local search procedure.

For the second question, let us assume that $f :\cM \rightarrow \R^{k+1}$ is sectioned and that we can find $D(t) \in \cM$ maximizing $\ell_t(x)$ for any $t \in \tDomain$. We want to find some $x^* \in \cM$ satisfying $f_1(x^*) = c_1, f_2(x^*) = c_2, \dots, f_k(x^*) = c_k$ and which maximizes $f_0(x^*)$, i.e. we want to find an $x^*$ which solves \ref{eq:constrained}. 
Indeed, we will see that it is at least possible to compute the value of $\Voptf(c)$ to arbitrary precision using oracle access to $D(t)$ for arbitrary values of $t$, using the ellipsoid algorithm. However, it is unclear whether we can recover the optimizer $x^*$ which witnesses the value of $\Voptf(c)$ in the same way, as we are only able to solve a `dual' problem, in a sense we will make precise in \Cref{subsec:ellipsoid}.

In order to get quantitative approximation guarantees for both the algorithmic tasks of finding $D(t)$ for $t \in \tDomain$ and solving \eqref{eq:constrained} given access to the value of $D(t)$,  
it will be necessary to be able to bound how quickly $D(t)$ changes depending on its input. The speed at which $D(t)$ changes depending on $t \in \tDomain$ will control how quickly we can make progress in a path tracking algorithm for computing $D(t)$, and also how closely we can approximate a solution to the constrained optimization problem \eqref{eq:constrained} using the ellipsoid algorithm.
To obtain such quantitative bounds, we will assume that the domain $\cM$ carries not merely the structure of a topological/metric space, but that of a \emph{Riemannian manifold.} This will allow us to control the Lipschitz parameter of the function $D(t)$ using calculus.

We will first provide the basic tools we will need from Riemannian geometry and optimization, based on the reference~\cite{boumal}, and in particular, we will state a result bounding the Lipschitz constant of $D(t)$ given control of certain condition numbers of the Riemannian Hessian of the Lagrangians.
In \Cref{subsec:CHORD}, we demonstrate an algorithm, which we call CHORD (recall that this is Continuous Homotopy Optimization with Riemannian Descent), that implements a path tracking methodology to solve Lagrangian dual problems.
In \Cref{subsec:ellipsoid}, we show how to make use of the ellipsoid algorithm to solve constrained problems of the form \eqref{eq:constrained} given access to an oracle solving the Lagrangian problems.
The CHORD method improves upon standard ideas from Riemannian optimization because it uses properties of the Lagrangian dual section to ensure that the output is close to a global optimizer.
By contrast, Riemannian gradient descent may become trapped in a local minimum with suboptimal objective value. 

Informally, CHORD is able to get global optimality guarantees because the Lagrangian dual section ensures that the global optima for nearby settings of the Lagrange multipliers are nearby. This implies if RGD run on an objective defined by $\langle t, f(x)\rangle$, and it is initialized at $D(t')$ for $t'$ close to $t$, then it will converge to $D(t)$. Formalizing this notion will require a manifold structure on $\cM$ and control over the Lipschitz parameter on $D(t)$.

\subsection{Background on Riemannian Optimization}
\label{subsec:riemannian}

In~\Cref{subsec:ellipsoid,subsec:CHORD}, we analyze two  algorithms that allow us to solve optimization problems with Lagrangian dual sections.
A shared requirement for these results is the need to quantify the continuity of the Lagrangian dual section $D$, i.e. its Lipschitz constant.

Recall that a \demph{smooth embedded submanifold} $\cM \subseteq \R^N$ of dimension $d$ may be defined as follows~\cite[Def 3.10]{boumal}: for every $x \in \cM$, there exists a neighborhood $U$ and a smooth function $g : \R^N \rightarrow \R^{N-d}$ so that $\cM \cap U = U \cap \{g(x) = 0\}$, and in such a way that the derivatives of coordinate functions $\nabla g_1(x), \dots, \nabla g_{N-d}(x)$ are linearly independent on $U$. When this holds, we denote by $T_x\cM$ the \demph{tangent space} of $\cM$ at $x$, defined as the $d$ dimensional vector space $\{y \in \R^N :\forall i = 1,\dots ,N-d,\; \langle \nabla g_i(x), y\rangle = 0\}$. 
Equipping each tangent space with the restriction of the standard inner product on $\mathbb{R}^N$ gives $\cM$ the structure of a \demph{Riemannian manifold.}
Our key examples $\cM_\lambda , \cM_\sigma ,$ and $\St^{n,m}$ are all embedded Riemannian submanifolds of Euclidean matrix spaces.
Some of the results that follow could be generalized to the setting of abstract Riemannian manifolds, e.g.~by invoking the famous Nash embedding theorem~\cite[Theorem 1, p52]{Nash1956-imbedding-problem}. 
Working with embedded manifolds affords us simpler descriptions of some technical notions from Riemannian geometry, such as the Riemannian Hessian defined below.

We will also assume that our function $f : \cM \rightarrow \R^{k+1}$ arises by restricting some smooth extension map $\overline{f}:\R^N \rightarrow \R^{k+1}$ to $\cM$. We may therefore also consider a smooth extension of the Lagrangian of $f$, $\overline{\cL} : \tDomainSmooth \times \R^N \rightarrow \R$ where $\overline{\cL}(t, x) = \langle t, \overline{f}(x)\rangle$. Finally, we will let $\overline{\ell_t}(x) = \overline{\cL}(t,x)$.

We may define the \demph{Riemannian gradient} of $\ell_t$,
\begin{align*}
\nabla_\cM \ell_t : \cM &\to T \cM , \\
\nabla_\cM \ell_t (x) &= P_x \nabla \overline{\ell_t} (x),
\end{align*}
a smooth section of the tangent bundle $T\cM = \sqcup_{x\in \cM} T_x \cM $ (i.e.~a \demph{vector field} on $\cM$). 
Here, the symbol $\nabla $ denotes the usual (Euclidean) gradient in $\R^n,$ and $P_x$ denotes orthogonal projection onto the tangent space $T_x \cM.$
As a $n\times n$ matrix,
\begin{align*}
P_x= I - \nabla g(x)^T \left( \nabla g(x)^T \nabla g (x) \right)^{-1} \nabla g (x),
\end{align*}
where $\nabla g (x): \R^n \to \R^s$ is usual (Euclidean) Jacobian of the local defining function $g$ at $x.$

The Riemannian gradient $\nabla_\cM \ell_t$ 
admits a \demph{smooth extension} $\overline{\nabla_\cM \ell_t} $ to a $N$-dimensional neighborhood $\cN \subset \R^N$ of $\cM$, i.e.~a smooth vector field $\overline{\nabla_\cM \ell_t}: \cN \to T \cN$ with $\overline{\nabla_\cM \ell_t} |_{\cM} = \nabla_\cM$.
By~\cite[Corollary 5.16]{boumal}, we have the \demph{Riemannian Hessian},
\begin{align}
\nabla_\cM^2 \ell_t (x)  : T_x \cM &\to T_x \cM \nonumber \\
v &\mapsto  P_x \cdot \nabla \overline{\nabla_\cM \ell _t} (x)  \cdot v.
\label{eq:riemannian-hessian}
\end{align}

For fixed $t,$ if $\ell_t$ is maximized at $x$, then we have a first derivative test~\cite[Prop.~4.5]{boumal},
\begin{equation}\label{eq:crit-points-t}
\nabla_\cM \ell_t (x) = 0,
\end{equation}
and a second derivative test~\cite[Prop.~6.3]{boumal},
\begin{equation*}\label{eq:crit-points-t2}
    \nabla_\cM^2 \ell_t (x) \preceq 0.
\end{equation*}

Now suppose $D$ is a Lagrangian dual section, and let $t^{(0)}, t^{(1)} \in \tDomain$. 
For simplicity, we will restrict our attention to the line segment joining $t^{(0)}$ to $t^{(1)}$, parameterized by $t(\lambda) = \lambda t^{(0)} + (1-\lambda)t^{(1)}$ for $\lambda \in \R$.

Abusing notation, let us write $D (\lambda ) = D\circ t (\lambda ).$
By the first derivative test, the following equations hold for all $\lambda \in [0,1]$ when $t^{(0)}$ and $t^{(1)}$ are sufficiently close:
\[
    \begin{pmatrix}
        \overline{\nabla_{\cM} \ell}_t(D(\lambda )\\
        g(D(\lambda ))
    \end{pmatrix} = 0.
\]
We now differentiate with respect to the scalar $\lambda $: writing $\ell_\lambda (x)= \ell_{t(\lambda )} (x),$
\[
    \begin{pmatrix}
        \nabla \overline{\nabla_{\cM} \ell_\lambda } \cdot D'(\lambda ) + \overline{\nabla_{\cM} \ell_\lambda }' (D(\lambda )) \\
        \nabla g  (D(\lambda )) \cdot D '(\lambda )
    \end{pmatrix} = 0,
\]
where the second equation is equivalent to $D'(\lambda ) \in T_{D(\lambda )} \cM.$

So long as the Riemannian Hessian is invertible, we have
\begin{equation}\label{eq:davidenko-manifold}
     D'(\lambda ) = -(\nabla_{\cM}^2 \ell_\lambda(D(\lambda )))^{-1}\cdot \nabla_{\cM} \ell_\lambda ' (D(\lambda )).
\end{equation}
Thus, $D(\lambda )$ in this case is (locally) the solution of a differential equation on $\cM $.
We use this observation to quantify the continuity of $D(t).$

\begin{lemma}
\label{lem:mean_value}
    With notation as above, let $s,t, \in \tDomainSmooth$, and $t(\lambda) = \lambda s + (1-\lambda) t$ for $\lambda \in [0,1]$.
    Supposing that for all $\lambda \in [0,1]$ we have both 
    \begin{equation}\label{eq:bound-hessian}
        \nabla_{\cM}^2 \ell_\lambda (D(\lambda )) \preceq \mu I, 
    \end{equation}
    and    \begin{equation}\label{eq:bound-gradient-derivative}
        \|\nabla_{\cM} \ell_\lambda  '(D(\lambda ))\| \le M,
    \end{equation}
   we then have 
    \begin{equation}\label{eq:bound-lagrangian-section-continuity}
        \|D(t) - D(s)\| \le \frac{M}{\mu}\|t-s\|.
    \end{equation}
\end{lemma}
\begin{proof}
We apply the mean value inequality in Euclidean space and~\eqref{eq:davidenko-manifold}, 
\begin{align*}
\| D (s) - D(t) \| &\le   \|t-s\| \cdot \max_{\lambda \in [0,1]} \|D ' (\lambda ) \|\\
&\le  \|t-s\| \cdot  \displaystyle\max_{\lambda \in [0,1]} \| ( \nabla_\cM^2 \ell_\lambda (D(\lambda))
)^{-1}\|_{\text{op}} \, 
\|
\nabla_\cM \ell_\lambda ' (D(\lambda)) \|,
\end{align*}
where 
$\| \bullet \|_{\text{op}}$ denotes the operator norm on the appropriate tangent space.
    
\end{proof}
\subsection{A Path Tracking Algorithm for Solving Lagrangian Optimization Problems}
\label{subsec:CHORD}

Given $\hat{x}_t \approx D(t)$ and a sufficiently small $\epsilon >0,$ the method of Riemannian Gradient descent (RGD) can be used to refine $\hat{x}_t$ as an approximation of $D(s)$ when $\| s - t \| < \epsilon .$
We explain how to make this effective.
In general, RGD depends on the choice of a \emph{retraction} on $\cM$: this is a smooth map
\begin{align*}
R : T\cM &\to \cM \\
(x,v ) &\mapsto R_x (v) 
\end{align*}
such that every curve $c(\lambda ) = R_x (\lambda v)$ satisfies $c(0)=x$ and $c'(0)=v.$
More generally, $R$ may be defined on a neighborhood of $\cM$ in $T \cM .$

Fix a retraction $R$ on $\cM$. We may define the \emph{$R$-injectivity radius} $\operatorname{inj}_R (x)$ at a point $x\in \cM$ is the supremum over all radii $r>0$ such that $R_x$ restricts to a diffeomorphism on 
\begin{equation}\label{eq:ball-Bxr}
B(x , r) = \{ v \in T_x \cM \mid \| v \| < r \}.
\end{equation}
The inverse function theorem implies $\operatorname{inj}_R (x) > 0$---see~\cite[Corollary 4.17]{boumal}.

Fix $x\in \cM$, $r < \operatorname{inj}_R (x)$, and consider the neighborhood 
\[
U := R_x (B(x,r)) \ni x.
\]
For $i=1,2,$  take $x_i \in U,$ and 
$v_i = R_x^{-1} (x_i)$. 
A convexity condition holds:
    \begin{equation}
    \label{eq:Rx-convex-curve}
    R_x ((1-\lambda) v_1 + \lambda  v_2) \in U 
    \quad 
    \forall \lambda \in [0,1].
    \end{equation}
When $R$ is the \emph{exponential map} on $\cM,$ the quantity $\operatorname{inj}_R (x)$ reduces to the standard notion of the injectivity radius, simply denoted $\operatorname{inj} (x)$
The condition~\eqref{eq:Rx-convex-curve} is equivalent to the statement that $U$ is a \emph{geodesically convex set} (\cite[Definition 11.2]{boumal}.)
Furthermore, our lower bound on the Riemannian Hessian implies that $\ell_t$ restricts to a \emph{geodesically convex function} on $U$ (cf.~\cite[Theorem 11.23]{boumal}.)

Let $L$ be a Lipschitz constant for the Riemannian gradient (\cite[Corollary 10.47]{boumal}),
\begin{equation}\label{eq:Lipschitz-riemannian-gradient-constant}
L = 
\displaystyle\max_{\lambda \in [0,1]} \| \nabla_\cM^2 \ell_\lambda  (D(\lambda)) \|_{op} . 
\end{equation}
As long as we have
\begin{equation}\label{eq:epsilon-inj-bound-exponential}
    \epsilon < 
(1/2) \max_{\lambda \in [0,1]} \operatorname{inj} (D(\lambda )),
    \end{equation}
   RGD with the exponential retraction and a fixed step size $(1/L)$ converges at an explicit rate given in~\cite[Theorem 11.29]{boumal}. 
   Taking $x^{(0)}=\hat{x}_t$ with subsequent RGD iterates $x^{(1)}, \ldots , x^{(k)},$ it is enough that we have
   \[
   \operatorname{dist} (x^{(k)}, D(s)) 
\le 
\dist (x^{(0)}, D(s))
\sqrt{\kappa (1 - 1/\kappa )^k } 
\le 
 2 \epsilon
\sqrt{\kappa (1 - 1/\kappa )^k } 
< \epsilon ,
\]
where $\operatorname{dist}$ is the Riemannian distance on $\cM ,$
and $\kappa = L / \mu > 1$ is the Hessian condition number.
This bound holds provided that
\begin{equation}\label{eq:rgd-iterations-bound-exponential}
k >
\displaystyle\frac{\log (1 / \kappa  )}{\log (1 - 1/\kappa )}
>
\displaystyle\frac{\log (1 / (4 \kappa ) )}{\log (1 - 1/\kappa )}
   \end{equation}
RGD iterations.
We summarize this discussion in the following theorem.
\begin{theorem}\label{thm:homotopy-alg}
Let $D: \mathbb{R} \to \cM$ be the restriction to the segment $t(\lambda) \in \tDomainSmooth,$ $\lambda \in [0,1]$, of a Lagrangian dual section for a family of smooth objective functions $(\ell_t )_{t\in \tDomainSmooth}$ on a smooth submanifold $\cM\subset \R^n.$
The CHORD algorithm~\ref{alg:chord}, with the given hypotheses, produces an $\epsilon$-approximation of $D(1)$ with 
\begin{equation}\label{eq:total-gd-iters-exponential}
\left\lceil 
\displaystyle\frac{\mu M}{\epsilon} \cdot \left( 1 + 
\displaystyle\frac{ \log (\mu / L)}{\log (1 - \mu / L )} \right)
\right\rceil 
\end{equation}
RGD steps when the accuracy parameter $\epsilon $ satisfies ~\eqref{eq:epsilon-inj-bound-exponential}.
\end{theorem}

\begin{remark}
As  $\kappa $ tends to infinity, Taylor approximation gives
\[
\displaystyle\frac{\log ( 1 / \kappa )}{\log (1- 1/\kappa ) } 
\approx 
\kappa \log (\kappa ), 
\quad 
\kappa \to \infty .
\]
Thus, the number of RGD steps in~\Cref{thm:homotopy-alg} is superlinear in the condition number.
\end{remark}

At a high-level,~\Cref{thm:homotopy-alg} shows that the global optimum of any sectioned problem on a Riemannian manifold can be computed algorithmically---at least when certain constants are known and when the exponential retraction is used.
\Cref{ex:upp-bunny} below shows that this approach also has practical benefits.

\begin{algorithm}
\caption{Continuous Homotopy Optimization with Riemannian Descent }\label{alg:chord}
\begin{algorithmic}
\Require \\
\begin{enumerate}
\item $D(\lambda )$, Lagrangian dual section for $\ell_t$ restricted to a segment $t(\lambda ),$ $\lambda \in [0,1]$
\item $\hat{x}_0\in \cM$ with $\operatorname{dist} (\hat{x}_0, D(0)) < \epsilon $
\item $\epsilon >0$ such that~\eqref{eq:epsilon-inj-bound-exponential} holds
\item $R$, the exponential retraction on $\cM$ 
\item $\mu ,$ $M$, and $L$ constants satisfying~\eqref{eq:bound-hessian},~\eqref{eq:bound-gradient-derivative},~\eqref{eq:Lipschitz-riemannian-gradient-constant}
\end{enumerate}
\Ensure $\hat{x}^{(1)}\in \cM$ with $\operatorname{dist} (\hat{x}^{(1)}, D(1)) < \epsilon $
 \State $\lambda \gets 0$
\While{$\lambda < 1$}
\State $k\gets 0$
\State $x^{(k)} \gets \hat{x}_{\lambda }$
\State $\lambda \gets \min ( \lambda + \epsilon / (\mu M), 1)$
\While{$k < 1 + \left\lceil 
\displaystyle\frac{ \log (4 (L / \mu ) )}{ \log (1 - \mu / L )}\right\rceil $ }
\State $k\gets k+1$
\State $x^{(k)} \gets R_{x^{(k)}} ( - (1/L) \nabla_{\cM}\ell_\lambda (x^{(k)}) ) $
\EndWhile
\State $\hat{x}_\lambda \gets x^{(k)}$
\EndWhile
\end{algorithmic}
\end{algorithm}

\begin{example}\label{ex:upp-bunny}
A natural test case for the CHORD algorithm is the UPP, which we have analyzed in detail for the $(n,m) = (3,2)$ case in~\Cref{sec:UPP}.

To recall, the UPP asks us to find, for a set of points represented by the columns of a matrix $U \in \R^{d \times n}$ and a corresponding set of points represented by $V \in \R^{d \times m}$, a rotation and projection which best maps $U$ onto $V$. In practice, $V$ is often a noisy version of a true projection of $U$ and the goal is to recover the rotation + projection matrix. Formally, we want to minimize $\|UX - V\|^2$ for $X \in \St^{n,m}$. 

In the case $(n,m) = (3,2)$, this can be visualized naturally; the columns of $U$ and $V$ can be plotted in $\R^3$ and $\R^2$ respectively, and we can see the result of a particular projection graphically. Pictorially, the goal is to find an orthogonal projection of the 3D object which best matches the 2D projection.

We can put this problem into our framework by noting that the minimizers of $\|UX - V\|^2$ are identical to the maximizers of $\langle A, XX^{\intercal}\rangle + \langle B, X\rangle$, where $A$ and $B$ are defined in \cref{lem:UPP_equivalent}. We can then define the parametric family of optimization problems $t\langle A, XX^{\intercal}\rangle + \langle B, X\rangle$, and think of this as being the Lagrangian associated to the function $f(X) = (\langle B, X\rangle, \langle A, XX^{\intercal}\rangle)$. Note that if $t = 0$, then the Lagrangian optimization problem is simply to maximize $\langle B, X\rangle$, which is achieved when $X= UV^{\intercal}$, where $B = U\Sigma V^{\intercal}$ is the singular value decomposition of $B$. Thus, in the case that $f$ is sectioned, we can apply our CHORD algorithm with $\hat{x}_0 = UV^{\intercal}$ with appropriate choices of the Lipschitz parameters.

In \cref{fig:CHORDExample}, we illustrate the output of our algorithm and compare that to the natural alternative of simply applying RGD directly to the objective. Note that in this example, RGD has found a critical point, which is not the global minimizer of the objective. By contrast, CHORD finds a projection which is a good approximation of the ground truth. Repeating this experiment with 100 random choices for the projection direction and the noise shows that in roughly 50\% of cases, RGD finds such a spurious critical point. 

\begin{figure}[!htb]
   \centering
   \includegraphics[width=1\linewidth]{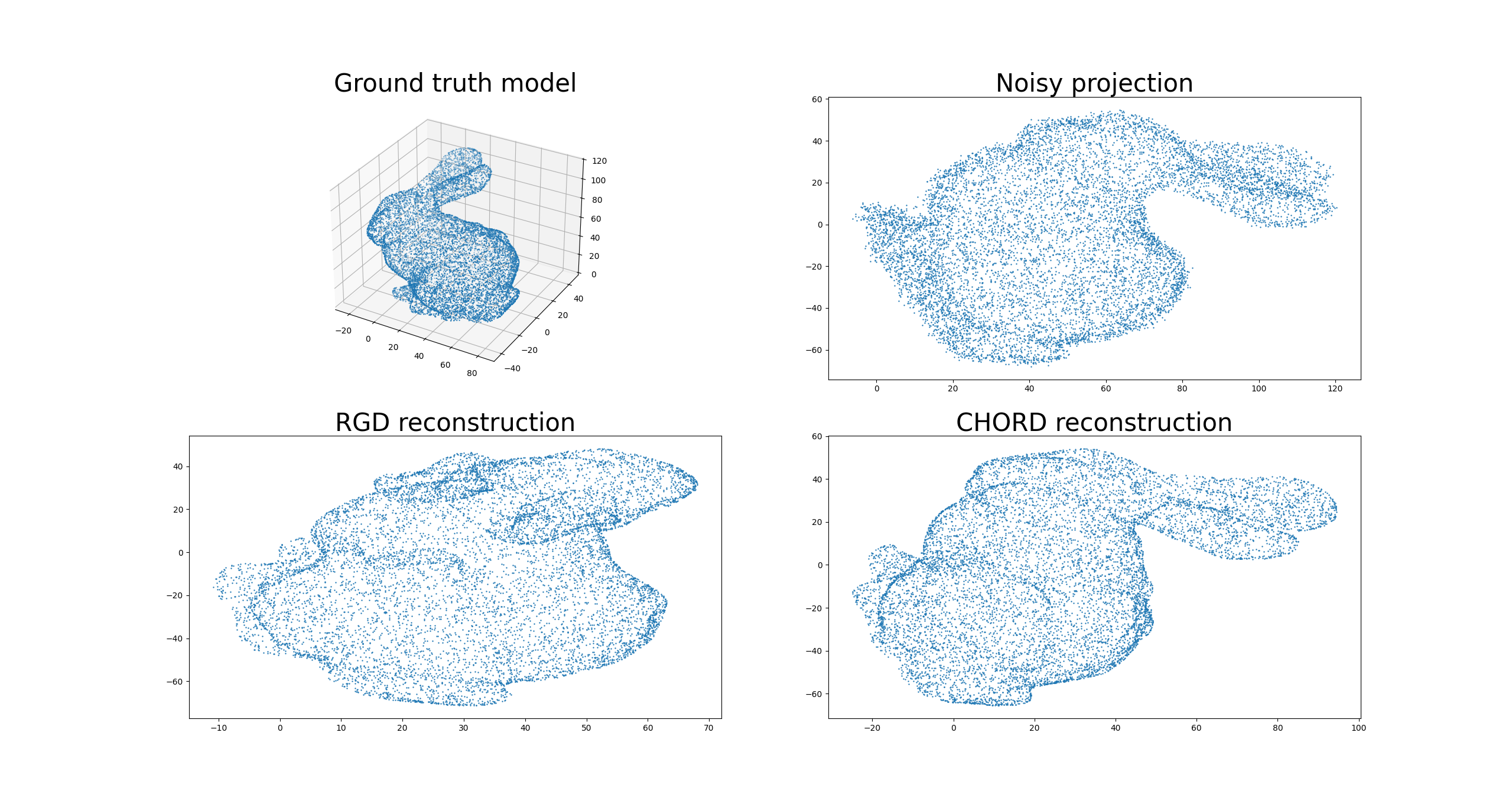}
   \caption{An example run of the CHORD algorithm. The top left shows a 3D data set, and the top right depicts a 2D projection of that data set where we have also added standard Gaussian noise. On the bottom left, we have a local minimum of the objective recovered by RGD with random initialization, and on the bottom right, we have the output of CHORD (using conservative estimates for the parameters). The 3D model is the Stanford bunny. For implementation of RGD, we used the package Pymanopt package~\cite{townsend2016pymanopt}, which uses a retraction based on $QR$ decomposition (a first-order approximation to the exponential retraction used in~\Cref{alg:chord}). The execution time of the CHORD algorithm with well tuned parameters is typically within a factor of 2 of the RGD algorithm.}
   \label{fig:CHORDExample}
\end{figure} 

\end{example}
\subsection{From Lagrangians to Constrained Optimization via the Ellipsoid Algorithm}
\label{subsec:ellipsoid}
Suppose that $f : \cM \rightarrow \R^{k+1}$ is sectioned, and let $\cC = \conv f(\cM)$. 
We have seen that the constrained optimization problem~\eqref{eq:constrained} may be reformulated as the convex optimization problem
\begin{equation}
\label{eq:primal}\tag{cr-Main}
    \max \{ y_0 : y_1 = c_1, \dots, y_k = c_k, y \in \cC\}.
\end{equation}
We will give a brief exposition about how to use the ellipsoid algorithm to solve a `dual' version of this problem, which only requires the ability to compute the value of $D(t)$ for some Lagrangian dual section of $f$; in particular, to solve the dual problem, no manifold assumption on $\cM$ is needed. It is only when we go to recover the actual primal solution to the optimization that we will need to assume that $\cM$ is a manifold.

If $0\in \cC$, then define the polar of $\cC$ as
\[
\cC^o := \{t \in \R^{k+1} : \forall y \in \cC, \;\langle t, y\rangle \le 1\}.
\]
Similarly, we define the \emph{dual optimization problem} to be 
\begin{equation}
\label{eq:dual_prob}\tag{opt-polar}
    \min \left\{\frac{1 - \sum_{i=1}^{k} c_i t_i}{t_0} : t \in \cC^o, t_0 > 0 \right\}.
\end{equation}

A consequence of
the bipolar theorem~\cite[Ch.~IV, (1.2)]{barvinok}
is that we have equality between \eqref{eq:primal} and \eqref{eq:dual_prob}, namely
\[
    \max \{y_0 : y_1 = c_1, \dots, y_k = c_k, y \in \cC\} = \min \left\{\frac{1 - \sum_{i=1}^{k} c_i t_i}{t_0} : t \in \cC^o, t_0 > 0\right\}.
\]
A \emph{separation} oracle for set $\cC^o$ is an algorithm, which on an input $t$ outputs either the fact that $t \in \cC^o$ or produces some $y \in \cC$ so that $\langle t, y \rangle > 1$.
We can see that having access to an oracle for computing $D(t)$ is enough to produce a separation oracle for $\cC^o$; on an input $t$, either output $t \in \cC^o$ if $\langle t, f(D(t)) \rangle \le 1$, and output $f(D(t))$ otherwise.
Given a separation oracle for a convex set, it is possible to approximately minimize an arbitrary convex function on that set in polynomial time using the ellipsoid algorithm, under some additional mild technical assumptions that we will not enumerate here---see e.g.~\cite[\S 5.2]{ben2001lectures}.

We now turn to the primal problem of finding $x \in \cM$ satisfying $f(x) = \begin{pmatrix} V_f(c) \\ c\end{pmatrix}$.
Indeed, if we could find an exact minimizer $t^*$ minimizing $\frac{1 - \sum_{i=1}^{k} c_i t_i}{t_0}$ over $\{t \in \cC^o : t_0 > 0\}$, then we have seen that $x^* = D(t^*)$ solves \eqref{eq:constrained}. However, in practice, we are not able to find an exact solution $t^*$, but rather some point $\hat{t}$ so that $\|t^* - \hat{t}\| \le \epsilon$ for an $\epsilon$ of our choosing. We would like to find a bound on $\epsilon$ such that the inequality $\|t^* - \hat{t}\| \le \epsilon$ implies that $\|D(t^*) - D(\hat{t})\| \le \delta$. The following lemma shows that when $\cM$ is a Riemannian manifold and $\ell_t$ satisfies some derivative bounds, we can obtain such a bound on $\|D(t^*) - D(\hat{t})\|$.

\begin{lemma}
    Suppose that $\cM \subseteq \R^N$ is a Riemannian manifold, and that $f : \R^N \rightarrow \R^{k+1}$ is smooth.
    Suppose that we are able to find $\hat{t} \in \Rplus^{k+1}$ satisfying $\|\hat{t} - t^*\| \le \epsilon$.
    Further suppose for all $t \in \Rplus^{k+1}$ with $\|t - \hat{t}\| \le \epsilon$,  that
    \[
        \nabla_{\cM}^2 \ell_t'(D(t)) \preceq \mu I,
    \]
    and
    \[
        \|\nabla_{\cM} \ell_t'(D(t))\| \le M.
    \]
    We then have that $\|D(\hat{t}) - D(t^*)\|\le \frac{M}{\mu}\epsilon$.
\end{lemma}
\begin{proof}
    This is an immediate consequence of \cref{lem:mean_value}.
\end{proof}

\section{Discussion and Future Directions}
We have given a general topological framework for analyzing whether a (possibly nonconvex) optimization problem has a natural convex reformulation. Using this framework, we showed broad connections between algebraic topology and convexity, and demonstrated the effectiveness of our notions in the context of some canonical nonconvex optimization problems. In particular, we showed connections between some classic optimization problems, such as QCQPs and Stiefel manifold optimization, and algebraic concepts such as noncrossing subspaces and Lie group theory. We also give some new results on the unbalanced orthogonal Procrustes problem, and some novel algorithmic ideas for tackling that problem.

In the future, we would like to extend the applicability of these results.

\paragraph{Approximate convexity} Our results in this paper are only concerned with whether certain problems have exact convex reparametrizations. 
We do not expect to find exact convex reparametrizations for all applications of interest.
Thus, a future direction is to obtain
\emph{quantitative results} which control the \emph{approximation quality} of a convex relaxation.
For example, 
a first step would be 
to show how \cref{thm:hidden_cvx} can be generalized to the setting where the continuous function 
$D : \tDomain \rightarrow \cM$  only approximately maximizes the Lagrangian. Further steps can be taken by understanding classic approximation results for convex relaxations through the Lagrangian dual section~lens.

\paragraph{Convex hierarchies} A classic tool for strengthening a convex relaxation is to consider lifting the problem to a larger number of variables and then imposing more constraints. For example, the sum-of-squares hierarchy or the Sherali-Adams hierarchy are well-known approaches for this task. A natural question is to ask whether our approach has any implications for how to choose such an extension, and whether or not these topological notions are helpful for making such a choice.
\section{Acknowledgments}
Research by Rodriguez was partially supported by the Alfred P. Sloan Foundation.
Research by Chandrasekaran and Shu was partially supported by AFOSR grant FA9550-23-1-0070.

\bibliographystyle{plain}
\bibliography{references}
\appendix
\section{Proofs of Standard Lemmas}
\label{app:proofs}
Here, we will prove some lemmas whose proofs are well known in their respective fields, but may be unfamiliar to some readers.

\begin{proof}[Proof of \cref{cor:convex_reform}]
By \cref{thm:hidden_cvx}, it suffices to show that if $V_f$ is concave, then 
\[
    V_f(c) = \max \{y_0 : y_1 = c_1, \dots, y_k = c_k, x \in \conv f(\cM)\}.
\]

    Fix $c \in \R^k$.
    It is clear that since $f(\cM) \subseteq \conv(f(\cM))$, 
    \begin{align*}
        V_f(c) &= \max \{f_0(x) : f_1(x) = c_1, \dots, f_k(x) = c_k, x \in \cM\}\\
        &\le
        \max \{y_0 : y_1 = c_1, \dots, y_k = c_k, y \in \conv(f(\cM))\},
    \end{align*}
    and so to show that these are equal, it suffices to prove the reverse inequality.

    It is clear from the definition of $V_f$ that
    \[
        f(\cM) \subseteq \{y \in \R^{k+1}: y_0 \le V_f(y_1, \dots, y_k)\},
    \]
    where this second set is convex by the concavity of $V_f(c)$.
    Standard properties of the convex hull imply
    \begin{equation}
    \label{eq:convex_containment}
        \conv f(\cM) \subseteq \{y \in \R^{k+1}: y_0 \le V_f(y_1, \dots, y_k)\},
    \end{equation}
    and, since $\barf$ is a linear function,
    \[
        \conv(f(\cM)) = \barf(\conv(\cM)).
    \]
    This last equality together with \eqref{eq:convex_containment} implies that, for any $x \in \conv(\cM)$,
    \[
        \barf_0(x) \le V_f(\barf_1(x), \dots, \barf_k(x)).
    \]
    Thus, we conclude that
    \begin{align*}
        \max \{y_0 : y_1 = c_1, \dots, y_k = c_k, y \in \conv(f(\cM))\} \le
        V_f(c).
    \end{align*}

\end{proof}
\begin{proof}[Proof of \cref{lem:homotopy_fact}]
We will use some calculations at the level of homology. We will refer to \cite{MR1867354} broadly for a definition of homology, and \cite[Cor.~2.15]{MR1867354} specifically, which carries out an identical computation.

Suppose that $\phi$ is not surjective, and that $y \in B^k$ specifically is not in the image of $\phi$. We would then have the following diagram of continuous maps commutes:
\[\xymatrix{
        B^k \ar[r]^{\phi}
& B^k \setminus y\ar@<.5ex>@{.>}[d] \\
        S^{k-1}\ar[u] \ar@<.5ex>@{->}[r]^{\partial \phi} &
    S^{k-1} \ar@<.5ex>@{->}[u] \ar@<.5ex>@{.>}[l]
}\]
Here the dotted arrows indicate the existence of a homotopy inverse. This induces a diagram of the homology groups $H_{k-1}$:
\[\xymatrix{
        H_{k-1}(B^k) \ar[r]^{\phi_*}
& H_{k-1}(B^k \setminus y)\ar@<.5ex>@{.>}[d] \\
        H_{k-1}(S^{k-1})\ar[u] \ar@<.5ex>@{->}[r]^{\partial \phi_*} &
    H_{k-1}(S^{k-1}) \ar@<.5ex>@{->}[u] \ar@<.5ex>@{.>}[l]
}
\]
The dotted arrows corresponding to homotopy equivalences induce isomorphisms of homology groups.
Simplifying this diagram, we obtain
\[\xymatrix{
        0 \ar[r]^{\phi_*}
& \Z\ar@<.5ex>@{.>}[d] \\
        \Z\ar[u] \ar@<.5ex>@{->}[r]^{\partial \phi_*} &
    \Z \ar@<.5ex>@{->}[u] \ar@<.5ex>@{.>}[l].
}\]
If we look at the map induced from the bottom left corner to the top right by going clockwise around the square, we obtain the zero map, while going counterclockwise, we obtain an isomorphism, which is a contradiction.

\end{proof}
\end{document}